\documentclass[10pt]{article}
\usepackage{amsmath,amssymb}
\usepackage{algorithm,algorithmic}
\usepackage{graphicx,subfigure}
\usepackage{cite}

\usepackage{amsthm}
\newtheorem{theorem}{Theorem}
\newtheorem{lemma}{Lemma}

\newtheorem{assumption}{Assumption}
\newtheorem{remark}{Remark}

\usepackage[capitalize,noabbrev]{cleveref}

\newcommand{\R}{\mathbb{R}}
\newcommand{\C}{\mathbb{C}}
\newcommand{\mc}[1]{\mathcal{#1}}
\newcommand{\mb}[1]{\mathbb{#1}}

\newcommand{\Ah}{\widehat{A}}
\newcommand{\Fh}{\widehat{F}}
\newcommand{\Uh}{\widehat{U}}
\newcommand{\Vh}{\widehat{V}}
\newcommand{\Sh}{\widehat{\Sigma}}

\newcommand{\Zh}{\widehat{Z}}
\newcommand{\uh}{\widehat{u}}
\newcommand{\vh}{\widehat{v}}

\newcommand{\define}{{\equiv}}
\newcommand{\rank}{\mathsf{rank}\,}
\newcommand{\trace}{\mathsf{trace}\,}

\newcommand{\norm}[2]{\left\lVert #1\right\rVert_{#2}}
\newcommand{\normf}[1]{\left\lVert #1\right\rVert_F}
\newcommand{\normtwo}[1]{\left\lVert #1\right\rVert_2}
\newcommand{\schattenp}[1]{{\left\vert\kern-0.25ex\left\vert\kern-0.25ex\left\vert #1 
    \right\vert\kern-0.25ex\right\vert\kern-0.25ex\right\vert}_p}
\newcommand{\uninorm}[1]{{\left\vert\kern-0.25ex\left\vert\kern-0.25ex\left\vert #1 
    \right\vert\kern-0.25ex\right\vert\kern-0.25ex\right\vert}}

\newcommand{\range}{\mc{R}\,}
\newcommand{\diag}{\mathsf{diag}\,}
\newcommand{\proj}[1]{\mathcal{P}_{#1}}

\newcommand{\bmat}[1]{\begin{bmatrix} #1 \end{bmatrix}}

\Crefname{prop}{Proposition}{Propositions}
\Crefname{assumption}{Assumption}{Assumptions}

\usepackage[top=1in, bottom=1in, left=1.25in, right=1.25in]{geometry}

\title{Randomized subspace iteration: Analysis of canonical angles and unitarily invariant norms}
\author{Arvind K. Saibaba\thanks{Department of Mathematics, North Carolina State University asaibab@ncsu.edu. This work was funded, in part, by NSF DMS 1720398, OP: Collaborative Research: Novel Feature-Based, Randomized Methods for Large-Scale Inversion}}
\begin{document}
\maketitle

\begin{abstract}
This paper analyzes the randomized subspace iteration for the computation of low-rank approximations. We present three different kinds of bounds. First, we derive both bounds for the canonical angles between the exact and the approximate singular subspaces. Second, we derive bounds for the low-rank approximation in any unitarily invariant norm (including the Schatten-p norm). This generalizes the bounds for Spectral and Frobenius norms found in the literature. Third, we present bounds for the accuracy of the singular values. The bounds are structural in that they are applicable to any starting guess, be it random or deterministic, that satisfies some minimal assumptions. Specialized bounds are provided when a Gaussian random matrix is used as the starting guess. Numerical experiments demonstrate the effectiveness of the proposed bounds. 
\end{abstract}


\section{Introduction}
The computation of low-rank approximations of large-scale matrices is a vital step in many applications in data analysis and scientific computing. These applications include principal component analysis, facial recognition, spectral clustering, model reduction techniques such as proper orthogonal decomposition (POD) and discrete empirical interpolation method (DEIM), approximation algorithms for partial differential  and integral equations. The celebrated Eckart-Young theorem~\cite{GovL13} says that the optimal low-rank approximation can be obtained by means of the Singular Value Decomposition (SVD); however, computing the full or truncated SVD can be computationally challenging, or even prohibitively expensive for many applications of interest.

{
Randomized algorithms for computing low-rank approximations have become increasingly popular in the last two decades. For example, see the survey papers~\cite{HMT09,mahoney2011randomized}. Randomized methods have gained in popularity since they are easy to implement, computationally efficient, and numerically robust. Although randomized algorithms tend to have the same asymptotic cost compared to classical methods, they have several advantages that make them suitable for large-scale computing. Specifically,  for  datasets that are too large to fit in memory, randomized algorithms are able to exploit parallel computing efficiently and are efficient in the number of times they access the data.  Randomized algorithms also have excellent numerical robustness and are very reliable in practical applications.

We focus on a specific randomized algorithm known as {\em randomized subspace iteration}. The main idea of this method is to use random sampling to identify a subspace that approximately captures the range of the matrix. A low-rank approximation to the matrix is then obtained by projecting the matrix onto this subspace. A post-processing step is then performed to compress the low-rank representation to achieve a desired target rank, and a conversion step to obtain an equivalent representation in the desired format (typically, a truncated SVD representation)---both these steps are deterministic. }

Many advances have been made in the analysis of randomized algorithms for low-rank approximations. The analysis typically has two stages: a \textit{structural, or deterministic stage}, in which minimal assumption about the distribution of the random matrix is made, and a \textit{probabilistic stage}, in which the distribution of the random matrix is taken into account to derive bounds for expected and tail bounds of the error distribution. As mentioned earlier, existing literature only targets the error in the low-rank representation~\cite{gu2015subspace,HMT09}. When the low-rank representation is in the SVD format, it is desirable to understand the quality of the approximate subspaces and the individual singular triplets. This paper aims to fill in some of the  missing gaps in the literature by a rigorous analysis of the accuracy of approximate singular values, vectors and subspaces obtained using randomized subspace iteration. This analysis will be beneficial in applications where an analysis beyond the low-rank approximation is desired. Examples include Model Reduction techniques~\cite{erichson2017randomized,balabanov2018randomized}, Leverage Score computation~\cite{holodnak2015conditioning}, Spectral Clustering~\cite{boutsidis2015spectral}, FEAST eigensolvers~\cite{peter2014feast}, Canonical Correlation Analysis~\cite{avron2013efficient}.

\subsection{Contributions and overview of paper} 
We survey the contents and the main contributions of this paper.

\textbf{Canonical angles.} We have developed bounds for \textit{all} the canonical angles between the spaces spanned by the exact and the approximate singular vectors. Several different flavors of bounds are provided: 
\begin{enumerate}
\item The bounds in \cref{ssec:angles} relate the canonical angles between the exact and the approximate singular subspaces. Analysis is also provided for unitarily invariant norms of the canonical angles.
\item In applications where lower dimensional subspaces are  extracted from the approximate singular subspaces, the bounds in \cref{ssec:sintheta} quantifies the accuracy in the extraction process. 
\item \cref{ssec:sintheta} also presents bounds for the angles between the individual exact and approximate singular vectors, extracted from the appropriate subspaces.
\end{enumerate}
Our bounds suggest that the accuracy of the singular values and vectors, in addition to the low-rank approximations, is high provided (1) singular values decay rapidly beyond the target rank $k$, and (2) the larger the singular value gaps, the higher is the accuracy to be expected. Furthermore, the truncation step to extract the $k$ dimensional subspaces does not significantly lower the accuracy of the subspaces.

\textbf{Low-rank approximation.} This paper provides the first known analysis of the randomized subspace iteration for an arbitrary unitarily invariant norm, with stronger, specialized results for Schatten-p norms.  Bounds for the special cases of the Schatten-p norm, namely the spectral and Frobenius norms, have already appeared in the literature---our result for the Schatten-p norm recovers these results as special cases.

\textbf{Singular values.} We derive upper and lower bounds on the approximate singular values obtained by the randomized subspace iteration. Similar bounds also appear in~\cite{gu2015subspace}; however, our proof technique is different. We also present Hoffman-Wielandt type bounds for the accuracy of the singular values. 

The conclusion of the bounds for the low-rank approximations and the singular values are similar to those of the conclusions for the canonical angles. 

\textbf{Generalization of sin theta theorem} The sin theta theorem~\cite{wedin1983angles} is a well known result in numerical analysis and relates the canonical angles between the true and approximate singular subspaces in the unitarily invariant norms. We derive a generalization of the sin theta theorem that derives bounds for the individual canonical angles between the two subspaces. The sin theta theorem is recovered as a special case. This result {maybe} of independent interest beyond the study of randomized algorithms.

\section{Background and preliminaries}
\subsection{Notation}\label{ssec:prelim}

{
Denote the target rank by $k$ and let $1 \leq k \leq \rank(A)$. Let the matrix $A\in \C^{m\times n}$, have the SVD 
\[A = \bmat{U_k & U_{\perp}}\bmat{\Sigma_k & \\ &\Sigma_{\perp} }\bmat{ V_k^* \\ V_\perp^*}. \]
Here, $\Sigma_k \in \C^{k\times k}$ and $\Sigma_{\perp} \in \C^{(m-k) \times (n-k)}$; the columns of $U_k$ and $U_{\perp}$ are the corresponding left singular vectors, and columns of $V_k$ and $V_{\perp}$ are the corresponding right singular vectors.  We also denote by $A_k = U_k\Sigma_kV_k^*$ as the best rank$-k$ approximation to the matrix $A$, in any unitarily invariant norm (for a definition, see below). We also define $A_\perp = U_\perp \Sigma_\perp V_\perp^*$ and observe that 
\[ A = A_k + A_\perp .   \]
  }

 \paragraph{Singular values and ratios} Let $\normtwo{\cdot}$ denote the spectral norm, so that $\normtwo{\Sigma_{\perp}} = \sigma_{k+1}$ and $\normtwo{\Sigma_k^{-1}} = \frac{1}{\sigma_k}$. The singular values of $A$ can be arranged in decreasing order as 
 \[ \sigma_1 \geq \sigma_2 \geq \dots \geq \sigma_k \geq \sigma_{k+1} \geq \dots \geq \sigma_n.\]
 For later use, we define the singular value ratios
 \begin{equation}
\gamma_j = \frac{\sigma_{k+1}}{\sigma_j} \qquad j=1,\dots,k.
\end{equation}
Since the singular values are monotonically decreasing, the singular value ratios are monotonically increasing, i.e.,  $\gamma_1 \leq \dots \leq \gamma_k \leq 1$.

 \paragraph{Norms} We have already defined the spectral norm. The Frobenius norm of a matrix is $\normf{A} = \sqrt{\trace(A^*A)}$. We use the symbol $\uninorm{\cdot}$ to denote any unitarily invariant norm, i.e., a norm that satisfies $\uninorm{QAZ} = \uninorm{A}$ for unitary matrices $Q,Z$. An example of the unitarily invariant norms is Schatten-p class of norms, defined as the vector $\ell_p$ norm of the singular values of $A$, i.e.,
\[ \schattenp{A} = \left( \sum_{j=1}^{\min\{m,n\}}\sigma_j^{p}\right)^{1/p}.\]
With this definition, it can be readily seen that $\normtwo{A}=  \uninorm{A}_{\infty}$ and $\normf{A} = \uninorm{A}_{2}$. Another example is the Ky-Fan-$k$ class of norms defined $\|A \|_{(k)} = \sum_{j=1}^k \sigma_j$ for every $k=1,\dots,\min\{m,n\}$. Associated with every unitarily invariant norm is a symmetric gauge function acting on the singular values of the matrix that it acts on.

 \paragraph{Projection matrices} Suppose the matrix $Z$ has full column rank with column space $\range(Z)$; $Z^\dagger$ is a left multiplicative inverse and where ${}^\dagger$ represents the Moore-Penrose inverse. We define the (orthogonal) projection matrix $\proj{Z} = ZZ^\dagger$. {An orthogonal projection matrix is uniquely defined by its range, and $\range(\proj{Z}) = \range(Z)$.} For a matrix $Q$ with orthonormal columns, the formula simplifies and $\proj{Q} = QQ^*$. 
 
\paragraph{Canonical angles} The separation between subspaces can be measured by the principal or canonical angles. Let $\mc{M}$ and $\mc{N}$ be two subspaces of $\mb{C}^n$, such that $\dim \mc{M} = \ell$, $\dim\mc{N} = k$ and $\ell\geq k$. Then the principal angles between the subspaces $\mc{M}$ and $\mc{N}$ are recursively defined to be the numbers $0 \leq \theta_i \leq \pi/2$ such that
\[ \cos\theta_i = \max_{u \in \mc{M}, v \in \mc{N} \\ \normtwo{u} = \normtwo{v} = 1} v^*u = v_i^*u_i, \qquad i = 1,\dots,k\]
subject to the constraints $\normtwo{u_i} = \normtwo{v_i} = 1$,  and 
\[u_j^*u, \quad v_j^*v = 0,  \qquad j=1,\dots,i-1.\] 
The canonical angles are arranged in increasing order as 
\[ 0 \leq \theta_1 \leq \dots \leq \theta_k \leq \pi/2.\] 
It can also be shown that $\sin\theta_i$ are also the singular values of $\proj{\mc{M}} - \proj{\mc{N}}$. 

We denote $\angle(\mc{M},\mc{N})$ to be the canonical angles between subspaces $\mc{M}$ and $\mc{N}$.  Let $M$ and $N$ be  matrices with orthonormal columns, which form bases for subspaces $\mc{M}$ and $\mc{N}$ respectively. Then, the singular values of  singular values of $(I-MM^*)N$ can be used to compute $\sin\angle(\mc{M},\mc{N})$ and the singular values of $M^*N$ can be used to compute $\cos\angle(\mc{M},\mc{N})$~\cite[Section 3]{bjorck1973numerical}. For ease of notation, in the rest of this paper, we write $\angle(M,N)$ instead of $\angle(\mc{M},\mc{N})$.

\subsection{Randomized subspace iteration} The basic version of the randomized subspace iteration is summarized in \cref{alg:randsvd}. Given a starting guess, denoted by $\Omega \in \C^{n\times (k+\rho)}$, the algorithm performs $q$ steps of the randomized subspace iteration to obtain the  matrix $Y$, also known as the ``sketch.'' A thin-QR factorization of $Y$ is performed to obtain $Q$ whose columns form an orthonormal basis for the range of $Y$. The main idea is that, under suitable conditions, the range of $Q$ is a good approximation for the range of $A$.  We obtain a low-rank approximation to $A$ by the projection $\Ah = QQ^*A$. The rest of the algorithm involves converting this low-rank approximation into the SVD format.
\begin{algorithm}
\begin{algorithmic}[1]
\REQUIRE Matrix $A$, Starting guess $\Omega \in \C^{n\times (k + \rho)}$, an integer $q \geq 0$.
\STATE Compute $Y = (AA^*)^qA\Omega$ 
\STATE Compute thin QR factorization of $Y$, so that $Y = QR$.
\STATE Compute $B = Q^*A$ and its SVD $B = U_{B}\Sh \Vh^*$. 
\STATE Compute $\Uh = QU_{B}$.
\RETURN Matrices $\Uh,\Sh,\Vh $ that define $\Ah \equiv \Uh\Sh\Vh^*$.
\end{algorithmic}
\caption{Idealized version of Subspace iteration for Singular Value Decomposition}
\label{alg:randsvd}
\end{algorithm}
 The algorithm to compute an approximate singular value decomposition, given starting guess $\Omega \in \mb{C}^{n\times (k+\rho)}$ is summarized in~\cref{alg:randsvd}. We say that this is  an idealized version, since the algorithm can behave poorly in the presence of round-off errors. A practical implementation of this algorithm alternates the QR factorization with matrix-vector products (matvecs) involving $A$; for more details regarding the implementation, the reader is referred to~\cite{saad2011numerical,HMT09}. In~\cref{alg:randsvd}, the output $$\Ah \equiv QQ^*A =  \Uh\Sh\Vh^*$$
may have a larger rank than (or equal to) $k$.  If a rank-$k$ approximation to $A$ is desired, then it can be obtained by discarding the $\rho$ smallest singular values of $\Ah$. We denote this low-rank representation by 
\[ \Ah_k = \Uh_k \Sh_k \Vh_k^*. \]
This is summarized in~\cref{alg:truncation}. 
 
 \begin{algorithm}[!ht]
 \begin{algorithmic}[1]
 \REQUIRE Matrix $A\in \C^{m\times n}$ and $Q \in \C^{m\times (k+\rho)}$. Target rank $1 \leq k \rank(A)$. 
 \STATE Form matrix $B = Q^*A$. 
 \STATE Compute the truncated SVD representation $B_k = \Uh_{B,k} \Sh_k \Vh_k^*$.
 \STATE Form $\Uh_k = Q\Uh_{B,k}$
 \RETURN Matrices $\Uh_k,\Sh_k,\Vh_k$  such that $\Ah_k = \Uh_k\Sh_k\Vh_k^*$.
 \end{algorithmic}
 \caption{Truncated SVD of $\Ah = QQ^*A$}
 \label{alg:truncation}
 \end{algorithm}

\noindent  {Before we state the assumptions needed for our analysis, we introduce the following notation. The matrix $V^*\Omega$ captures the influence of the starting guess on the right singular matrix $V$. Partition this matrix as
\begin{equation}\label{eqn:omegadef} V^*\Omega = \bmat{V_{k}^*\Omega \\ V_{\perp}^*\Omega} = \bmat{\Omega_1 \\ \Omega_2},\end{equation}
where $\Omega_1 = V_k^*\Omega \in \C^{k\times (k+\rho)}$ and $\Omega_2 = V_\perp^*\Omega \in \C^{(n-k)\times (k+\rho)}$. 
As was mentioned earlier, we assume that the  target rank $k$ satisfies $1 \leq k \leq \rank(A)$. Additionally, the following assumptions will be required for our analysis.
\begin{assumption}\label{ass:main} 
Let $\Omega_1 \in \C^{k\times (k+\rho)}$ be defined as above. We assume that 
\begin{equation}\label{eqn:omega1} \rank(\Omega_1) = k.  \end{equation}  
\item The singular value gap at index $k$ is inversely proportional to the singular value ratio
\begin{equation}
\gamma_k = \normtwo{\Sigma_\perp}\normtwo{\Sigma_k^{-1}}= \frac{\sigma_{k+1}}{\sigma_k} < 1.
\end{equation}
\end{assumption}
The first assumption guarantees that the starting guess $\Omega$ has a significant influence over the right singular vectors, whereas the second assumption ensures that the $k$ dimensional subspace $\range(U_k)$ is well defined. In practice, it is highly desirable that $\gamma_k \ll 1$, which ensures that there is a large singular value gap.}

\section{Accuracy of singular vectors}\label{sec:angles}

We want to understand how well $\range (\Uh)$ approximates $\range(U_k)$, measured in terms of the canonical angles between the subspaces. To this end, abbreviate the subspace angles between $\Uh \in \C^{m\times \ell}$ and $U_k \in \C^{m \times k}$ as $\theta_1,\dots,\theta_k$. Similarly, denote the angles between $\Vh \in \C^{n\times \ell}$ and $V_k \in \C^{n \times k}$ by $\nu_1,\dots,\nu_k$. We are also interested in obtaining bounds for the canonical angles $\angle(U_k,\Uh_k)$ and $\angle(V_k,\Vh_k)$. To distinguish these angles from $\angle(U_k,\Uh)$ and $\angle(V_k,\Vh)$, we call them $\theta_j'$ and $\nu_j'$ for $j=1,\dots,k$. 
\subsection{Bounds for canonical angles}\label{ssec:angles}

Our first result derives bounds for the canonical angles $\angle(U_k,\Uh)$. The analysis is based on the perturbation of projectors and the tools used here are similar to~\cite{HMT09}. 

\begin{theorem}\label{thm:u}
 Let $\Uh$ and $\Vh$ be obtained from~\cref{alg:randsvd}. With~\cref{ass:main}, the canonical angles $\theta_j$ and $\nu_j$ satisfy
\[ \sin\theta_j\leq \frac{\gamma_{j}^{2q+1} \normtwo{\Omega_2\Omega_1^\dagger}}{\sqrt{1 + \gamma_{j}^{4q+2} \normtwo{\Omega_2\Omega_1^\dagger}^2}}  \qquad \sin\nu_j \leq \frac{\gamma_{j}^{2q+2} \normtwo{\Omega_2\Omega_1^\dagger}}{\sqrt{1 + \gamma_{j}^{4q+4} \normtwo{\Omega_2\Omega_1^\dagger}^2}}   \]
 for $j=1,\dots,k$. 
\end{theorem}•  
{
This theorem has several interesting features worth pointing out. First, if the matrix has exact rank $k$, then all of the canonical angles are uniformly equal to zero; that is, the randomized subspace iteration identifies the subspace exactly.  On the other hand, when $\gamma_k$ is very close to $1$, the subspaces may not be well-defined and {may be} difficult to identify. In practice, it is highly desirable that $\gamma_k \ll 1$, so that the angles are captured accurately.

Second, the bounds for the canonical angles show explicit dependence on the singular value ratios $\gamma_j$. In particular, the canonical angles $\theta_j$ and $\nu_j$ converge to zero quadratically but at different rates depending on the singular value ratios.  Specifically, the smaller the singular value ratio, smaller the canonical angles.

Third, the term $\normtwo{\Omega_2\Omega_1^\dagger}$ can be written in terms of the right singular vector matrix $V$ and the starting guess $\Omega$ as 
\[ \normtwo{\Omega_2\Omega_1^\dagger} = \normtwo{ (V_{\perp}^*\Omega)(V_k^*\Omega)^\dagger }.\]
When the columns of $\Omega$ is linearly independent, this quantity is nothing but the tangent of the largest canonical angle between $\range(V_k)$ and $\range(\Omega)$. This term appears frequently in randomized linear algebra and can be interpreted as  a measure of the subspace overlap between the starting guess and the right singular vectors. In the ideal case, $\Omega$ contains the singular vectors in $V_k$. 
A discussion of the meaning and interpretation of this term, is provided in~\cite[Section 2.5]{drineas2017structural}. In particular, when $\Omega$ is a Gaussian random matrix, $\normtwo{\Omega_2\Omega_1^\dagger}$ is roughly on the order of $\sqrt{(n-k)k}$. 

Fourth, the influence of $\normtwo{\Omega_2\Omega_1^\dagger}$ is subdued by the singular value ratios $\gamma_j^{2q+1}$.  With sufficiently large number of iterations $q$, the canonical angles are smaller than a user-defined tolerance. Rigorous bounds for the requisite number of iterations are provided in~\cref{ssec:prob}.

Lastly, the bounds for the canonical angles  $\theta_j$ are smaller than $\nu_j$ because the latter contains an additional power of $\gamma_j$. The reason for this higher accuracy is as follows: the columns of $\Vh$ are the right  singular vectors of $Q^*A$. Therefore, the multiplication step with $Q$ amounts to an additional step of subspace iteration and gives the extra factor. 
}

\begin{remark}
\cref{thm:u} gives the sine of the canonical angles; these bounds can also be used to obtain upper bounds for the tangents and lower bounds for the cosines. 
 With the same assumptions and notation as in~\cref{thm:u}, the relationship between the tangent and sine implies 
 \[\tan\theta_j \leq {\gamma_{j}^{2q+1} \normtwo{\Omega_2\Omega_1^\dagger}} \qquad \tan\nu_j \leq {\gamma_{j}^{2q+2} \normtwo{\Omega_2\Omega_1^\dagger}} \]
 for $j=1,\dots,k$. Lower bounds for cosine of the canonical angles follow similarly.
\end{remark}

\paragraph{Unitarily invariant norms}  The following result derives bounds for the canonical angles in any unitarily invariant norm, in contrast to~\cref{thm:u} which bounds the individual canonical angles. 

\begin{theorem}\label{thm:theta2F} Let the approximate singular vectors $\Uh$ and $\Vh$ for a matrix $A$ be computed according to~\cref{alg:randsvd}. Under~\cref{ass:main}, for every unitarily invariant norm, 
\begin{equation}
\begin{aligned}
 \uninorm{\sin\angle(U_k,\Uh) } \leq & \>  \gamma_k^{2q} \frac{ \uninorm{ \Sigma_\perp}}{\sigma_k} \normtwo{\Omega_2\Omega_1^\dagger}, \\ 
 \uninorm{\sin\angle(V_k,\Vh) } \leq & \> \gamma_k^{2q+1} \frac{\uninorm{ \Sigma_\perp}}{\sigma_k} \normtwo{\Omega_2\Omega_1^\dagger}. 
\end{aligned}
\end{equation}
\end{theorem}

The interpretation of this theorem is similar to that of~\cref{thm:u}. The connection between the two theorems follows from the identity $\sin\theta_k = \|\sin\angle(U_k,\Uh) \|_{2}$. If we specialize the result in~\cref{thm:theta2F} to the spectral norm, then it is clear that this result weaker than the bound in~\cref{thm:u}. 

\subsection{Extraction of $k$-dimensional subspaces}\label{ssec:sintheta}
In the previous subsection, the columns of $\Uh$ and $\Vh$ spanned $\ell = k +\rho$ dimensional subspaces. Many applications, however, require the extraction of $k$ dimensional singular subspaces from the low-rank approximation $\Ah \equiv QQ^*A$. One way to extract the appropriate subspaces is to first compute the optimal rank-$k$ truncation of $\Ah$, denoted by $\Ah_k$. The singular vectors of $\Ah_k$, denoted by $\Uh_k$ and $\Vh_k$, are then used instead of $\Uh$ and $\Vh$. See~\cref{alg:truncation}, for details regarding implementation. The bounds derived in the previous subsection are not directly applicable since~\cite[Corollary 10]{ye2016schubert} says 
\[ \theta_j \leq \theta_j'  \qquad \nu_j \leq \nu_j' \qquad j=1,\dots,k.\] 
To understand how much additional error is incurred during this extraction process, we present several results. The important conclusion of all these results is that the accuracy of the extracted subspaces of dimension $k$ is comparable to the accuracy of the $k+\rho$ dimensional subspace provided the singular values are sufficiently well separated.

The approach we take is different from that of the previous section. The starting point of our analysis is the well-known sin theta theorem for singular subspaces~\cite{wedin1983angles}. Let $A,\Ah$ be two matrices of conformal dimensions. Assuming that  
\begin{equation}\label{eqn:gap}
\zeta \equiv \sigma_k(A) - \sigma_{k+1}(\Ah) > 0,
\end{equation}
we have
\begin{equation}\label{eqn:sintheta} \max\left\{ \uninorm{\sin \angle(U_k,\Uh_k)} ,\uninorm{\sin\angle(V_k,\Vh_k)} \right\} \leq \frac{\max\{\uninorm{E_{12}},\uninorm{E_{21}}\}}{\zeta},
\end{equation}
 where  the two matrices $E_{12}$ and $E_{21}$ are 
\begin{equation} \label{eqn:e1}
\begin{aligned}
E_{12}  = & \>(I-\proj{\Uh_k}) (A-\Ah) \proj{V_k}\\
E_{21}  = &  \> \proj{{U}_k} (A-\Ah) (I-\proj{\Vh_k}).
\end{aligned} 
\end{equation}

However, this version of the sin theta theorem does not provide us with a way to obtain bounds for the individual canonical angles. To this end, we first present a new generalization of the sin theta theorem. 
\begin{theorem}\label{thm:gensintheta} Let $A \in \C^{m\times n}$ with $\rank(A) \geq k$ and let $\Ah$ be the perturbed matrix with same dimensions. Suppose the singular value gap satisfies~\cref{eqn:gap}. 
Let $\Ah_k = \Uh_k\Sh_k\Vh_k^*$ be the truncated SVD of $\Ah$. Then 
\[ \max \{ \sin \theta_j' , \sin \nu_j' \} \leq \frac{\sigma_k(A)}{\sigma_{j}(A)} \max\{\sin\theta_k',\sin\nu_k'\} \qquad j=1,\dots,k.\]
\end{theorem}
This theorem states that the sine of the canonical angles $\sin \theta_j'$ are bounded by  $\sin \theta_k' $ up to a multiplicative factor, which is at most $1$.

Our main result provides the following bounds for canonical angles between the exact and the approximate singular subspaces, when both the subspaces have the same dimension. The proof involves simplifying every term in~\cref{eqn:e1}.
\begin{theorem}\label{thm:sintheta} 
Let $\Uh$ and $\Vh$ be obtained from~\cref{alg:randsvd}, and matrices $\Uh_k$ and $\Vh_k$ from~\cref{alg:truncation}. Under~\cref{ass:main}, 
\begin{itemize}
\item for every unitarily invariant norm 
\[  \max\left\{ \uninorm{ \sin \angle (U_k,\Uh_k) } ,\uninorm{\sin\angle(V_k,\Vh_k)} \right\} \leq \phi\frac{\gamma_k^{2q}}{1-\gamma_k} \frac{\uninorm{\Sigma_\perp}}{\sigma_k} \normtwo{\Omega_2\Omega_1^\dagger}.\]
The factor $\phi$ takes different values depending on the specific norm used. For an arbitrary unitarily invariant norm, we have  $\phi= \sqrt{2}$, whereas for the spectral and Frobenius norms, we have $\phi = 1$.
\item  canonical angles $\theta_j'$ and $\nu_j'$ satisfy
\[\max \{ \sin \theta_j' , \sin \nu_j' \} \leq \gamma_{j} \frac{\gamma_k^{2q}}{1-\gamma_k}\normtwo{\Omega_2\Omega_1^\dagger} \qquad j=1,\dots,k.\]
\end{itemize}• 
\end{theorem}
The interpretation of this theorem is: (1) as the number of iterations $q$ increase,  the largest canonical angle converges to $0$ quadratically, and (2) a larger singular value gap means that the subspace is computed more accurately.  Comparing this result with~\cref{thm:theta2F}, we see that the upper bound in~\cref{thm:sintheta} has additional factors which depend on the specific norm used. For an arbitrary unitarily invariant norm, there is an additional factor $\max\{1,\sqrt{2}\gamma_k\}/(1-\gamma_k)$.  For the spectral and Frobenius norms, the additional factor is $1/(1-\gamma_k)$.  Both factors are greater than $1$, suggesting that the truncation process can introduce additional error. The additional factor is also independent of the number of iterations $q$, suggesting that it is a one-time price to be paid for the extraction process. The  bound is devastating when $\gamma_k \approx 1$, but this also means that the subspaces may not be well-defined.

\paragraph{Individual singular vectors}
The previous results give insight into the accuracy measured using the canonical angles between the exact and approximate singular subspaces. When individual singular vectors need to be extracted, does the extraction process introduce additional error? The following result quantifies the accuracy of the extraction process.
\begin{theorem}\label{t_singleuv}
Let the approximate singular vectors $\Uh$ and $\Vh$ be computed according to~\cref{alg:randsvd}. With~\cref{ass:main}, we have the following inequalities  
\begin{equation}\label{eqn:singleuv}
 \sin \angle(u_j,\Uh)\leq   {\gamma_{j}^{2q+1} \normtwo{\Omega_2\Omega_1^\dagger}} \qquad 
\sin\angle(v_j,\Vh) \leq  {\gamma_{j}^{2q+2} \normtwo{\Omega_2\Omega_1^\dagger}}
\end{equation}
for $j=1,\dots k$. Denote the approximate singular triplets $(\hat\sigma_j,\hat{u}_j,\hat{v}_j)$ for $j=1,\dots,k$. Under~\cref{ass:main}
\begin{equation}\label{eqn:svs}
 \max\left\{\sin\angle(u_j,\hat{u}_j),\sin\angle(v_j,\hat{v}_j) \right\} \leq \sqrt{1 + 2\frac{\tilde{\gamma}^2}{\tilde\delta^2}} \, \gamma_j^{2q+1}\normtwo{\Omega_2\Omega_1^\dagger}.
\end{equation}
Here, $\tilde\gamma^2 \equiv \normtwo{\Sigma_\perp}^2 + \normtwo{\Sigma_\perp\Omega_2\Omega_1^\dagger}^2 $ and $\tilde\delta \equiv \min\{\min_{\tilde{\sigma}_i\neq \tilde\sigma_j}\{ |\sigma_j - \tilde\sigma_i| ,\sigma_j\} \} $.
\end{theorem}• 
The first result bounds the angles between the exact singular vector and the corresponding approximate singular subspaces. The second result compares the angles of the exact and the approximate singular vectors. This result also says that the extraction process does not adversely increase the error in the singular subspaces, provided the singular values are well-separated. 
 
The convergence of the individual singular vectors tell a similar story to that of~\cref{thm:u}. The singular vectors corresponding to the largest singular values converge earlier than the singular vectors corresponding to the smaller singular vectors. This is a consequence of the fact that the singular value ratios are non-decreasing. 

\subsection{Comparison with other bounds}

The subspace iteration dates to a 1957 paper by Bauer~\cite{Bauer1957} for eigenvalue problems. The analysis of the subspace iteration has also been well-established, for example, we refer to~\cite[Chapter 14]{Par80}. {Randomized subspace iteration has attracted a lot of attention in the last two decades, with a special emphasis on quantifying the influence of the starting guess $\Omega$. In particular, recent research has focused on the choice of the distribution and the effect of the oversampling parameter $\rho$.  The effect of randomized subspace iteration on the accuracy of singular vectors was studied in the context of spectral clustering in~\cite{boutsidis2015spectral}. However, the authors made the rather strong assumption that $\Omega \in \R^{n\times k}$, which amounts to setting the oversampling parameter $\rho = 0$. This is a strong requirement since~\cref{ass:main} now requires $\Omega_1$ to be invertible.  The authors were able to show (in our notation)
\[ \normtwo{ \sin \angle (U_k,\Uh_k)} \leq \frac{\gamma_k^{2q+1}\normtwo{\Omega_2\Omega_1^{-1}}}{\sqrt{1 + \gamma_k^{4q+2}\normtwo{\Omega_2\Omega_1^{-1}}^2}}.\]
Notice that this bound coincides with~\cref{thm:u} (for $\sin\theta_k$) when $\rho = 0$. Our results provide bounds for the right singular vectors as well as all the canonical angles.

Let us return to this assumption that $\rank(\Omega_1)$. When $\Omega$ is standard Gaussian matrix,~\cite[Theorem 3.3]{sankar2006smoothed} says
 $$\|\Omega_1^{-1}\|_2 \leq \frac{2.35 \sqrt{k}}{\delta}$$
 with probability at least $1-\delta$. For a small probability of failure $0 <\delta < 1$, this bound can be devastating. By contrast, if we let $\Omega_1 \in \C^{k\times (k+\rho)}$ with $\rho \geq 2$, and still suppose that $\Omega$  is a Gaussian random matrix. Then, with probability at least $1-\delta$~\cite[Proposition 10.4]{HMT09} says
$$\normtwo{\Omega_1^\dagger} \leq e\frac{\sqrt{k+\rho}}{\rho} \left(\frac{1}{\delta}\right)^{1/(\rho+1)}. $$ 
It is clear that when the random matrix is Gaussian, oversampling  has an impact on the accuracy of the randomized subspace iteration. Specifically, larger the oversampling, the more accurate is the subspace.  }

Oversampling plays a bigger role for random matrices that have different distributions than Gaussian. When $\Omega$ is generated from the subsampled randomized Hadamard transform (SRHT), or Rademacher distributions, a more aggressive form of oversampling $\ell \sim k\log k$ is necessary to ensure that $\rank(\Omega_1) = k$. Therefore, by allowing for oversampling, our bounds are applicable to starting guesses that are not restricted to Gaussian random matrices. Not only that, our bounds are also informative for matrices with decaying singular values and significant singular value gap.

A recent paper by Nakatsukasa~\cite{nakatsukasa2017accuracy} considered the issue of accuracy of extracting singular subspaces for general projection-based approximation methods. In our notation, these refer to relating bounds for $\angle(U_k,\Uh)$ to $\angle(U_k,\Uh_k)$.  Our bounds for the canonical angles appear to be tighter than the result implied by~\cite[Corollary 1]{nakatsukasa2017accuracy}. This may be because the analysis was applicable to arbitrary subspace projections, whereas ours is specialized to randomized subspace iteration; we do not go into a detailed comparison here. Furthermore, our analysis is able to bound the individual canonical angles which is missing in~\cite{nakatsukasa2017accuracy}.

\subsection{Probabilistic bounds}\label{ssec:prob} Thus far, we
have not made specific assumptions on the matrix $\Omega$, as long as it
satisfies $\rank(\Omega_1) = k$. In particular, $\Omega$ need not be even be random, and
may be deterministic. However, more can be said about the bounds when $\Omega$ is random
is drawn from a specific distribution.

In many applications, the matrix $\Omega \in \R^{n\times (k+\rho)}$ is taken to be the standard Gaussian random matrix. That is, the entries of $\Omega$ are i.i.d.\ $\mc{N}(0,1)$ random variables. Here we derive a few probabilistic results that provide insight into the accuracy of the subspaces. Let $\rho \geq 2$ and define the constant
\begin{equation}
\label{eqn:ce}
 C_e = \sqrt{\frac{k}{\rho-1}} + \frac{e\sqrt{(k+\rho)(n-k)}}{\rho}
 \end{equation}
and for $0 < \delta < 1$ define the constant
\begin{equation}
\label{eqn:cd}
C_d = \frac{e\sqrt{k+\rho}}{\rho+1}\left(\frac{2}{\delta}\right)^{1/(\rho+1)} \left( \sqrt{n-k} + \sqrt{k+\rho} + \sqrt{2\log\frac{2}{\delta}}\right).
\end{equation}
\begin{theorem}[Probabilistic bounds]\label{thm:can_gauss} Let $\Omega \in \mb{R}^{n\times (k+\rho)}$ be a standard  Gaussian random matrix with $\rho \geq 2$. Assume that the singular value ratio $\gamma_k < 1$. Let $\Uh$ and $\Vh$ be obtained from~\cref{alg:randsvd}. For $j=1,\dots,k$, the expected value of the canonical angles satisfy
\[ \mb{E}\, \left[\sin\theta_j \right] \leq \frac{\gamma_{j}^{2q+1}C_e}{\sqrt{1 + \gamma_{j}^{4q+2}C_e^2} } \qquad \mb{E} \, \left[\sin\nu_j \right] \leq \frac{\gamma_{j}^{2q+2}C_e}{\sqrt{1 + \gamma_{j}^{4q+4}C_e^2} }.\]
Let $0 < \delta < 1$ be a user defined failure tolerance. With probability, at least $1-\delta$, the following inequalities hold independently for $j=1,\dots,k$
\[  \sin\theta_{j} \leq \frac{\gamma_{j+1}^{2q+1}C_d}{\sqrt{1 + \gamma_{j}^{4q+2}C_d^2} } \qquad \sin\nu_j \leq \frac{\gamma_{j}^{2q+2}C_d}{\sqrt{1 + \gamma_{j}^{4q+4}C_d^2} }.\]
 \end{theorem}

The main message of theorem can be seen from the following bound on the number of subspace iterations $q$.  Specifically, suppose $0 < \epsilon < 1$, and the number of subspace iterations $q$ we take satisfies 
\[ q \geq \frac12 \left( \frac{\log \epsilon/C_e}{\log \gamma_k} -1\right),\]
then $  \mb{E}\, \sin \theta_j \leq \mc{O}(\epsilon^2)$ for $j=1,\dots,k$.

 Several extensions of these results are possible. First, following the proof technique of~\cref{thm:can_gauss}, we can extend the probabilistic analysis to~\cref{thm:theta2F,t_singleuv} as well. Second, following the strategy in~\cite{HMT09}, the probabilistic results can be extended to other distributions. However, we will not pursue these extensions here.

\section{Low-rank approximation and Singular values}\label{sec:aux}

In this section, we provide several structural bounds for the accuracy of the low-rank approximation and the accuracy of the singular values. 

\subsection{Low-rank approximation} Several results are available for
estimating the error in the low-rank approximation $A \approx QQ^*A$ in the
spectral and Frobenius norms, when the matrix $Q$ is obtained from the
randomized subspace
iteration~\cite{HMT09,gu2015subspace,zhang2016randomized}. As was mentioned
earlier, the spectral and Frobenius norms are special cases of the Schatten-p
norm, which are examples of unitarily invariant norms.

Here we present the first known analysis of randomized subspace iteration in a unitarily invariant norm. 
\begin{theorem}\label{thm:lowrank_schatten}
Let $\Ah \in \C^{m\times n}$ be computed using~\cref{alg:randsvd}. Under~\cref{ass:main}, the following inequalities hold in every unitarily invariant norm 
\begin{align}\label{eqn:nnorm}
 \uninorm{(I-QQ^*)A} \leq & \> \uninorm{\Sigma_\perp}+ \gamma_k^{2q}\uninorm{\Sigma_\perp\Omega_2 \Omega_1^\dagger } \\ \label{eqn:nnorm_rankr}
\uninorm{(I-QQ^*)A_k } \leq & \>  \gamma_k^{2q}\uninorm{\Sigma_\perp\Omega_2\Omega_1^\dagger}.
\end{align}•
Let $B =Q^*A$, and let $B_k$ be its best rank$-k$ approximation. If $A$ is approximated using $QB_k$, then the error in the low-rank approximation is
\begin{equation}\label{eqn:nnorm_rankr2}
\uninorm{A-QB_k} \leq  \> \left( 1+ \frac{\sigma_1}{\sigma_k}\frac{\phi\gamma_k^{2q}}{1-\gamma_k} \normtwo{\Omega_2 \Omega_1^\dagger} \right)\uninorm{\Sigma_\perp}.\end{equation}
 As in~\cref{thm:sintheta}, $\phi = 1$ for spectral and Frobenius norms, and $\sqrt{2}$ for an arbitrary unitarily invariant norm. 
\end{theorem}
In this theorem, as the number of iterations $q\rightarrow \infty$, the error in the low-rank approximation goes to zero.

We present a variant of the error in the low-rank approximation for the special case that a Schatten-p norm is used.  The proof for the special case of the Frobenius norm was provided in~\cite{zhang2016randomized}. 
\begin{theorem}\label{thm:lowrank}
Let $\Ah$ be computed using~\cref{alg:randsvd}. Under~\cref{ass:main}, we have
\begin{equation}\label{eqn:2norm}
\schattenp{ (I-QQ^*)A }^2 \leq \schattenp{\Sigma_\perp}^2 + \gamma_k^{4q}\schattenp{\Sigma_\perp\Omega_2\Omega_1^\dagger}^2.
\end{equation}•
\end{theorem}

The error bound in~\cref{thm:lowrank_schatten} is weaker than~\cref{thm:lowrank} for the Schatten-p norm  since for $\alpha,\beta \geq
0$, we have $\sqrt{\alpha^2 + \beta^2} \leq \alpha + \beta$. More generally,~\cref{thm:lowrank} is applicable to any unitarily invariant norm that is
also a Q-norm~\cite[Definition IV.2.9]{bhatia1997matrix}. A unitarily invariant norm $\uninorm{\cdot}_Q$ is a $Q$-norm, if there exists
another unitarily invariant norm $\uninorm{\cdot}_a$ such that $\uninorm{A}_Q^2 =
\uninorm{A^*A}_a$. Note that the Schatten-p norms satisfy this property for $p
\geq 2$, since $\schattenp{A}^2 = \uninorm{A^*A}_{p/2}$.

\subsection{Accuracy of singular values} How are the singular values of $A$
related to the singular values of $\Ah$? We now present a result that
quantifies the accuracy of the individual singular values. This result is
similar to~\cite[Theorem 4.3]{gu2015subspace}. Our proof techniques are
substantially different. We make extensive use of the Cauchy interlacing
theorem and the multiplicative singular value inequalities~\cref{eqn:singprod}.
 
\begin{theorem}\label{thm:sigma}
Let $\Ah = \Uh\Sh\Vh^*$ be computed using~\cref{alg:randsvd}. Under~\cref{ass:main}, the
approximate singular values $\sigma_j(\Ah)$ satisfy for $j=1,\dots,k$ 
\[ \sigma_j(A) \geq \sigma_j(\Ah) \geq \frac{\sigma_j(A)}{\sqrt{1 + \gamma^{4q+2}_{j}\normtwo{\Omega_2\Omega_1^\dagger}^2 }}.\]
\end{theorem}•
It can be readily seen that the large singular values are computed more accurately since the singular value ratio corresponding to larger singular values is smaller. 

Rather than quantify the accuracy of the individual singular values, the next results are of the Hoffman-Wielandt type and account for all the singular values together. Define the two matrices of conformal sizes
\[ \Sigma = \bmat{\Sigma_k \\ & \Sigma_\perp} \qquad \Sigma' = \bmat{\Sh \\ & 0}.\]
Under~\cref{ass:main}, the error in the singular values satisfies
\begin{equation}
 \uninorm{\Sigma-\Sigma'} \leq \uninorm{\Sigma_\perp}  + \gamma_k^{2q} \uninorm{\Sigma_\perp\Omega_2\Omega_1^\dagger} .
\end{equation}
The proof combines~\cite[III.6.13]{bhatia1997matrix} with~\cref{thm:lowrank_schatten}. For the Schatten-p norm, with $p\geq 2$, we can derive the bound
\begin{equation}\schattenp{\Sigma-\Sigma'} \leq \sqrt{\schattenp{\Sigma_\perp}^2  + \gamma_k^{4q} \schattenp{\Sigma_\perp\Omega_2\Omega_1^\dagger}^2} .
\end{equation}
The proof is similar, and is therefore omitted.

\section{Proofs}\label{sec:proofs}
We recall some results here that will be useful in our analysis, see~\cite[Section 7.7]{HoJ13} for proofs.  Let $M,N$ be Hermitian positive definite. The notation $M \preceq N$ means $N-M $ is positive semi-definite and it defines a partial ordering on the set of Hermitian matrices. Clearly, this also implies $I -N \preceq I -M$. The partial order is preserved under the conjugation rule. That is
\[ SMS^* \preceq SNS^* \qquad \forall \> S \in \mb{C}^{m\times n}. \]
Weyl's theorem implies that the eigenvalues satisfy  $\lambda_j(M) \leq \lambda_j(N)$ for all $j=1,\dots,n$. If additionally, $M,N$ are both positive semidefinite then $M^{1/2} \preceq N^{1/2}$~\cite[Proposition V.1.8]{bhatia1997matrix} and $(I+N)^{-1} \preceq (I+M)^{-1}$.  

  \paragraph{Singular value inequalities} Let $A,B \in \C^{m\times n}$. For all $i,j$ such that $1 \leq i,j \leq \min\{m,n\}$ and $i+j-1\leq \min\{m,n\}$, the following singular value inequalities hold for the sum $A+B$~\cite[Equation 7.3.13]{HoJ13}
 \begin{equation}
\sigma_{i+j-1}(A+B) \leq \sigma_i(A) + \sigma_j(B),
\end{equation}
and product $AB^*$~\cite[Equation (7.3.14)]{HoJ13}
\begin{equation}\label{eqn:singprod}
\sigma_{i+j-1}(AB^*)\leq \sigma_i(A)\sigma_j(B).
\end{equation}•
 A useful corollary of these results is that $\sigma_i(A+B) \leq \sigma_i (A) + \sigma_1(B)$ and $\sigma_i(AB^*) \leq \sigma_i(A) \sigma_1(B)$ for $i=1,\dots,\min\{m,n\}$.

\paragraph{Unitarily invariant norms} {It is useful to recall some properties of the unitarily invariant norms. Every unitarily invariant norm $\uninorm{\cdot}$ on $\C^n$ is associated with a symmetric gauge function on $\R^n$. The  $\uninorm{\cdot}$  satisfies $\uninorm{M} = \uninorm{(M^*M)^{1/2}}$, since both matrices have the same nonzero singular values. The following inequality for unitarily invariant norms, also known as strong sub-multiplicativity, will be useful~\cite[(IV.40)]{bhatia1997matrix}  
$$\uninorm{ABC} \leq \normtwo{A}\normtwo{C}\uninorm{B}.$$

We will need the following lemma
 \begin{lemma}\label{lemma:schatten}
 Let $A, B, D \in \C^{n\times n} $ such that  $A,B$ Hermitian and  $0 \preceq A \preceq B$, then 
 \[ \uninorm{(D^*AD)^{1/2}} \leq \uninorm{(D^*BD)^{1/2}}.\]
 \end{lemma}
 
 \begin{proof}
 Combining the properties of the partial ordering, the eigenvalues of the scaled matrices satisfy $\lambda_j(D^*AD)^{1/2} \leq \lambda_j(D^*BD)^{1/2}$ for all $j=1,\dots,n$. Since the matrices are positive semidefinite, the eigenvalues are the singular values and $\|(D^*AD)^{1/2}\|_{(k)} \leq \|(D^*BD)^{1/2}\|_{(k)}$ for every Ky-Fan-k norm $k=1,\dots,n$. By the Fan dominance theorem~\cite[Theorem IV.2.2]{bhatia1997matrix}, the advertised inequality is true for every unitarily invariant norm.
 \end{proof}
}

\subsection{Proofs of~\cref{ssec:angles} Theorems}
\begin{proof}[\cref{thm:u}] We tackle each case separately. \\
\textbf{Bounds for $\sin\theta_j$}: The proof is lengthy and proceeds in four steps. We give a great level of detail here, since the proof technique will be applicable to the subsequent proofs.
\paragraph{1. Converting an SVD to an EVD} We compute the thin SVD of $(I-\proj{\Uh})U_k = K S_U G^*.$ The matrix $$S_U  = \diag(\sin\theta_k,\dots,\sin\theta_1) \in \mb{R}^{k\times k}$$ 
contains the sine of the canonical angles between the subspaces spanned by the columns of $\Uh$ and $U_k$~\cite[Equation (13)]{bjorck1973numerical}. It is readily seen that
 \begin{equation}\label{eqn:su} G S_U^2 G^* = U_k^*(I-\proj{\Uh})U_k.\end{equation}

\paragraph{2. Shrinking space} In~\cref{alg:randsvd}, we had defined $Y = (AA^*)^qA\Omega$. It follows that 
\[ U^*Y = \bmat{\Sigma_k^{2q+1} \\ & (\Sigma_\perp\Sigma_\perp^\top)^q\Sigma_\perp} (V^*\Omega)= \begin{bmatrix} \Sigma_k^{2q+1} \Omega_1 \\ (\Sigma_\perp\Sigma_\perp^\top)^q\Sigma_\perp \Omega_2\end{bmatrix},\]
where from~\cref{eqn:omegadef}, $\Omega_1 = V_k^*\Omega$ and $\Omega_2 = V_\perp^*\Omega$. Next, by~\cref{ass:main}, $\Omega_1$ has full row rank and therefore it has a right multiplicative inverse. Define
\[ Z \equiv U^*Y\Omega_1^\dagger \Sigma_k^{-(2q+1)} = \bmat{I \\ F}• \qquad F \equiv (\Sigma_\perp\Sigma_\perp^\top)^q\Sigma_\perp\Omega_2\Omega_1^\dagger\Sigma_k^{-(2q+1)}. \]
{Recall that  $Y= QR$ is the thin-QR factorization of $Y$. Let $Q_1R_1$ be the thin-QR factorization of $R\Omega_1^\dagger\Sigma_k^{-(2q+1)}$; here, $Q_1 \in \mb{C}^{(k+\rho) \times k}, R_1\in \mb{C}^{k\times k}$.

From $Q_1Q_1^*\preceq I$, the conjugation rule implies 
\[
\proj{Z} = U^*QQ_1Q_1^*Q^*U \preceq U^*QQ^*U = \proj{U^*Q}.
\]
Since
$\range(U^*Y) = \range(U^*Q) = \range(U^*\Uh)$, they have the same projectors, so 
\begin{equation}\label{eqn:projz}\proj{Z} \preceq \proj{U^*\Uh} \qquad I - \proj{U^*\Uh} \preceq I - \proj{Z}.\end{equation}

Plug in $UU^* = I$ into~\eqref{eqn:su}, and use~\eqref{eqn:projz} to obtain
\[ U_k^*(I-\proj{\Uh})U_k = U_k^*U (I-\proj{U^*\Uh}) U^*U_k  \preceq  \bmat{ I & 0 }(I-\proj{Z}) \bmat{I \\ 0}.\]
}

\paragraph{3. Simplifying $\proj{Z}$} 
Since  $\proj{Z} = ZZ^\dagger$, we have
\[ \proj{Z} = \bmat{I \\ F} (I + F^*F)^{-1} \bmat{I & F^*}, \]
from which, it can be readily seen that 
\begin{align} \nonumber
 \bmat{ I & 0 }(I-\proj{Z}) \bmat{I \\ 0} = & \> I - (I+F^*F)^{-1} \\ \label{eqn:H}
= &  \> F^*F (I+F^*F)^{-1}\equiv H.
\end{align}
Note that $H$ is positive semidefinite. To summarize the story so far, $GS_U^2 G^* \preceq  H$. 
\paragraph{4. Applying singular value inequalities}  A straightforward SVD argument shows that the $j$-th singular value of $ H$ satisfies
\[  \sigma_j(H) = \sigma_j^2(F)/(1+\sigma_j^2(F)) \quad j=1,\dots,k. \] 
The singular value inequalities~\cref{eqn:singprod} imply 
\[ \sigma_j(F) \leq \sigma_1( \Sigma_\perp\Sigma_\perp^\top)^q\Sigma_\perp\Omega_2\Omega_1^\dagger)\sigma_j(\Sigma_k^{-2q-1}) \leq \left(\frac{\sigma_{k+1}}{\sigma_{k-j+1}}\right)^{2q+1} \normtwo{\Omega_2\Omega_1^\dagger} .\]
Plugging this inequality into $\sigma_j(H)$ 
\[ \sigma_j^2(H) \> \leq  \>  \frac{\gamma_{k-j+1}^{4q+2} \normtwo{\Omega_2\Omega_1^\dagger}^2}{1 + \gamma_{k-j+1}^{4q+2} \normtwo{\Omega_2\Omega_1^\dagger}^2}\, \qquad j=1,\dots,k. \]
Since $GS_U^2 G^* \preceq  H$, Weyl's theorem implies $\sin^2 \theta_{k-j+1} \leq \sigma_j^2(H)$. 
Take square roots on both sides and rename $j \leftarrow k-j+1$ to get the desired result.

\textbf{Bounds for $\sin\nu_j$}:  Let $GS_V^2G^*$ be the eigenvalue decomposition of $V_k^*(I-\proj{\Vh})V_k$. Note that the diagonals of $S_V$ are the sine of the canonical angles $\angle(V_k,\Vh)$. Since $\Vh$
is obtained from the thin SVD of  $A^*Q$, $\range(A^*Q) =
\range(\Vh)$ and $\proj{\Vh} = \proj{A^*Q}$, since an orthogonal projection
matrix is uniquely determined by the range. Next, consider $\Zh$ defined as 
\begin{equation}\Zh \equiv \Sigma^\top U^*Y\Omega_1^\dagger \Sigma_k^{-2q-2} = \begin{bmatrix}I \\ \Fh\end{bmatrix}• \qquad \Fh \define (\Sigma_\perp^\top\Sigma_\perp)^{q+1}\Omega_2\Omega_1^\dagger\Sigma_k^{-2q-2}, \label{zhat}
\end{equation}• 
from $(AV)^*Q= \Sigma^*U^*Q $, it can be verified that 
\[ \range(\Zh ) \subset \range(\Sigma^\top U^*Y) = \range(\Sigma^\top U^*Q) = \range((AV)^*Q). \]
Using an argument similar to~\cref{eqn:projz}, we obtain
\[ V_k^*V(I-\proj{\Vh})V^* V_k   \preceq V_k^*V(I-\proj{\Zh})V^* V_k = \bmat{I & 0} (I-\proj{\Zh}) \bmat{I \\0}. \] 
The right hand side simplifies to $I - (I +\Fh^*\Fh)^{-1}$. The rest of the proof is similar to that of the proof for $\sin\theta_j$.
\end{proof}


\begin{proof}[\cref{thm:theta2F}]
With the notation of \cref{thm:u}, we follow steps 1-3 of the proof to obtain $$GS_U^2 G^* \preceq  H \preceq F^*F.$$
Since the square root preserves partial ordering, implies $GS_U G^* \preceq   (F^*F)^{1/2}$. {Note that $(F^*F)^{1/2}$ and $F$ have the same nonzero singular values. Therefore, $$\uninorm{\sin\angle(U_k,\Uh)} \leq\uninorm{(F^*F)^{1/2}} = \uninorm{F}.$$ 
By using strong sub-multiplicativity of the unitarily invariant norm, we have
$$\uninorm{\sin\angle(U_k,\Uh)} \leq \gamma_k^{2q}\normtwo{\Omega_2\Omega_1^\dagger} \frac{\uninorm{\Sigma_\perp}}{\sigma_k}.$$
} 
\end{proof}

\subsection{Proofs of~\cref{ssec:sintheta} Theorems}


\begin{proof}[\cref{thm:gensintheta}]
Let $X = (I-\proj{\Uh_k})\proj{U_k}$ and $Y = (I-\proj{\Vh_k})\proj{V_k}$. In decreasing order, the singular values of $X$ and $Y$ are $\{ \sin \theta_j'\}_{j=1}^k$ and $ \{\sin \nu_j'\}_{j=1}^k$ respectively. Let $B \equiv \Ah - \Ah_k$. First, we observe that
\begin{align*}
E_{12} = & \> (I-\proj{\Uh_k})(A - \Ah)\proj{V_k} \\ 
= & \> (I-\proj{\Uh_k})A_k - (I-\proj{\Uh_k}) \Ah \proj{V_k} \\
= &\> (I-\proj{\Uh_k})\proj{U_k}A_k - (\Ah - \Ah_k)\proj{V_k} \\
=& \> XA_k - B(I-\proj{\Vh_k})\proj{V_k} =XA_k - BY.
\end{align*}•
A similar calculation shows that $E_{21} = X^*B - A_kY^*$. From the first relation, since $\rank(A) \geq k$, we have 
\[ XA_k A_k^\dagger = (E_{12} + BY) A_k^\dagger.\]
But $A_k A_k^\dagger = \proj{U_k}$ and $X\proj{U_k} = X$. Applying~\cref{eqn:singprod}, we have  
\[ \sigma_j(X) \leq (\|E_{12}\|_2  + \|B\|_2\|Y\|_2)/\sigma_{k-j+1}(A) \qquad j=1,\dots,k.\]
A similar argument gives
\[\sigma_j(Y) \leq (\normtwo{E_{21}}  + \normtwo{B}\normtwo{X})/\sigma_{k-j+1}(A) \qquad j=1,\dots,k. \]
Combining these relations 
\[ \max \{ \sigma_j(X), \sigma_j(Y)\} \leq \frac{\max\{ \|E_{21}\|_2, \|E_{12}\|_2\}}{\sigma_{k-j+1}(A)}  + \frac{\|B\|_2}{\sigma_{k-j+1}(A)} \max\{\|X\|_2,\|Y\|_2\}. \]
Recognize that $\|B\|_2 = \sigma_{k+1}(\Ah)$. 
Applying~\cref{eqn:sintheta} in the spectral norm simplifies the expression since 
\[ \frac{1}{\sigma_{k-j+1}(A)}\left(1  + \frac{\sigma_{k+1}(\Ah)}{\sigma_{k}(A) - \sigma_{k+1}(\Ah)}\right) = \frac{\sigma_k(A)}{\sigma_{k-j+1}(A)(\sigma_{k}(A) - \sigma_{k+1}(\Ah))}. \]
Therefore, 
\[\max \{ \sigma_j(X), \sigma_j(Y)\} \leq \frac{\sigma_k(A)}{\sigma_{k-j+1}(A)}\frac{\max\{ \normtwo{E_{21}}, \normtwo{E_{12}}\}}{\zeta}.\]
Now $\sigma_j(X) = \sin\theta_{k-j+1}'$ and $\sigma_j(Y) = \sin\nu_{k-j+1}'$. Rename $j\leftarrow k-j+1$ to finish. 
\end{proof}

\begin{proof}[\cref{thm:sintheta}] We tackle each case independently. \\
\textbf{Unitarily invariant norms}: Our proof involves simplifying each term in~\cref{eqn:sintheta}, and~\cref{eqn:e1} and has several steps.
\paragraph{1. Simplifying the gap} Recall $\zeta = \sigma_k(A) - \sigma_{k+1}(\Ah)$ and $\Ah = QQ^*A$. From the first part of~\cref{thm:sigma} 
\[ \zeta = \sigma_k(A) - \sigma_{k+1}(\Ah) \geq \sigma_k(A) - \sigma_{k+1}(A).\]
	\paragraph{2. Simplifying $\uninorm{E_{12} }$} First observe that $A\proj{V_k} = A_k$. So $$E_{12} =(I-\proj{\Uh_k})(I-QQ^*)A\proj{V_k} =  (I-\proj{\Uh_k})(I-QQ^*)A_k.$$ Then applying~\cref{eqn:nnorm_rankr} along with sub-multiplicativity gives
\[ \uninorm{E_{12} } \leq \uninorm{(I-QQ^*)A_k}  \leq \gamma_k^{2q}\uninorm{\Sigma_\perp\Omega_2\Omega_1^\dagger } \leq \gamma_k^{2q} \uninorm{\Sigma_\perp}  \normtwo{\Omega_2\Omega_1^\dagger }. \]
\paragraph{3. Simplifying $\uninorm{E_{21} }$} First, $E_{21} = \proj{U_k}(I-QQ^*)A\proj{\Vh_k}$, and since $\normtwo{\proj{U_k}(I-QQ^*)} = \normtwo{\sin\angle(U_k,\Uh)}$, 
\[ \uninorm{E_{21} } \leq \normtwo{\sin\angle(U_k,\Uh)} \uninorm{ (I-QQ^*)A},\]
because of strong sub-multiplicativity. Applying \cref{thm:u} and~\cref{eqn:nnorm_rankr}
\[  \uninorm{E_{21}} \leq \frac{\gamma_k^{2q+1} \normtwo{\Omega_2\Omega_1^\dagger}}{\sqrt{1+\gamma_k^{4q+2} \| \Omega_2\Omega_1^\dagger\|_2^2}} \left(1  + \gamma_k^{2q}\normtwo{\Omega_2\Omega_1^\dagger} \right)\uninorm{\Sigma_\perp}.\]
Let $\beta = \gamma_k^{2q}\normtwo{\Omega_2\Omega_1^\dagger}$. Then for $\beta \geq 0$, since $\gamma_k < 1$
\[ \frac{\gamma_k(1+\beta)}{\sqrt{1+\gamma_k^2\beta^2}} \leq \frac{1+\gamma_k\beta}{\sqrt{1+\gamma_k^2\beta^2}} \leq \sqrt{2}. \] 
Therefore, $\uninorm{E_{21}} \leq \sqrt{2}\gamma_k^{2q} \uninorm{\Sigma_\perp} \|\Omega_2\Omega_1^\dagger\|_2  $.
\paragraph{4. Putting everything together} Plugging in the intermediate quantities into~\cref{eqn:sintheta}, we have
 \[ \max\left\{ \uninorm{ \sin \angle(U_k,\Uh_k)} ,\uninorm{\sin\angle(V_k,\Vh_k)} \right\} \leq \sqrt{2}\gamma_k^{2q} \normtwo{\Omega_2\Omega_1^\dagger} \frac{ \uninorm{\Sigma_\perp}}{\sigma_{k}-\sigma_{k+1}}. \]
Dividing the numerator and denominator by $\sigma_k$ proves the stated result for unitarily invariant norms. 

\textbf{Spectral/Frobenius norms}: Let $\norm{\cdot}{\xi}$ denote the spectral and Frobenius norms. The first two steps are identical to the proof for unitarily invariant norms. For the third step, using~\cref{eqn:nnorm_rankr}
\[ \norm{E_{21}}{\xi} \leq \frac{\gamma_k^{2q+1} \| \Omega_2\Omega_1^\dagger\|_2}{\sqrt{1+\gamma_k^{4q+2} \| \Omega_2\Omega_1^\dagger\|_2^2}} \norm{\Sigma_\perp}{\xi} \sqrt{1+\gamma_k^{4q} \normtwo{\Omega_2\Omega_1^\dagger}^2}.\]
With $\beta$ defined as before, since $\gamma_k < 1$,  ${\sqrt{\gamma_k^2+\gamma_k^2\beta^2}}/{\sqrt{1+\gamma_k^2\beta^2}} \leq  1$.
Therefore, $$\norm{E_{21}}{\xi} \leq \gamma_k^{2q} \norm{\Sigma_\perp}{\xi}\normtwo{\Omega_2\Omega_1^\dagger}.$$ The rest of the proof is the same.

\textbf{Canonical angles}: The proof combines~\cref{thm:gensintheta} with the above analysis for the spectral norm. The right hand side contains the term 
$$\max\left\{ \normtwo{ \sin \angle(U_k,\Uh_k)} ,\normtwo{\sin\angle(V_k,\Vh_k)} \right\}.$$ 
The rest of the proof involves some simple manipulations.
\end{proof}


\begin{proof}[\cref{t_singleuv}]
We first address~\cref{eqn:singleuv}. Following the steps of the proof of~\cref{thm:u}, we have 
\[ \sin^2 \angle (u_j,\Uh) = u_j^*U(I-\proj{U^*Q})U^*u_j \preceq \bmat{e_j^\top & 0 } (I - \proj{Z}) \bmat{e_j \\ 0}, \]
where $e_j$ is the $j$--th  column of the $k\times k$ identity matrix. Therefore, we have $\sin^2 \angle (u_j,\Uh) \leq e_j^\top He_j$,
where $H$ was defined in~\cref{eqn:H}. The inequality $H \preceq F^*F$ implies
\begin{align*}
  \sin^2 \angle (u_j,\Uh)  
 \leq& \>\sigma_{j}^{-4q-2}  \normtwo{(\Sigma_\perp\Sigma_\perp^\top)^q\Sigma_\perp(\Omega_2\Omega_1^\dagger)e_j}^2 \\
 \leq & \>  \gamma_{j}^{4q+2}\normtwo{\Omega_2\Omega_1^\dagger}^2. 
\end{align*}
Taking square-roots on both sides gives the desired results. The strategy for bounding the canonical angles $\sin\angle (v_j,\Vh)$ is very similar and will be omitted.

We now address~\cref{eqn:svs}, which is a straightforward application of~\cite[Theorem 2.5]{hochstenbach2004harmonic}. Let $\proj{\mc{U}} = QQ^*$ and $\proj{\mc{V}} = I$. Then, in our notation, this result takes the form
\[ \max\left\{\sin\angle(u_j,\uh_j),\sin\angle(v_j,\vh_j)\right\} \leq \sqrt{1 + 2\frac{\tilde\gamma'^2}{\tilde\delta^2}}\max\left\{\sin\angle(u_j,\Uh),\sin\angle(v_j,I)\right\}. \]
where $\tilde\gamma' = \max\{0, \normtwo{(I-QQ^*)A}\}$ and $\tilde\delta$ is as defined in the statement of the theorem. \cref{thm:lowrank} for the spectral norm implies $\tilde\gamma'\leq \tilde\gamma$, whereas \cref{t_singleuv} implies 
\[ \max\left\{\sin\angle(u_j,\Uh),\sin\angle(v_j,I)\right\} \leq \gamma_j^{2q}\normtwo{\Omega_2\Omega_1^\dagger}.\]
Plug in the intermediate steps to obtain the desired bound.  
\end{proof}


\begin{proof}[\cref{thm:can_gauss}]
In~\cref{thm:u}, bounds for $\normtwo{\Omega_2\Omega_1^\dagger}$ are available in the literature. From the proof of~\cite[Theorem 10.6]{HMT09} we find the inequality
$$\mb{E}\,\normtwo{\Omega_2\Omega_1^\dagger} \leq C_e,$$ where the constant $C_e$ was defined in~\cref{eqn:ce}. Let $\alpha > 0$ be a constant. The map $x \mapsto x/\sqrt{1+\alpha x^2}$ is convex. Therefore, by Jensen's inequality the results in expectation follow. 

For the concentration inequalities,~\cite[Theorem 5.8]{gu2015subspace} showed that $\normtwo{\Omega_2\Omega_1^\dagger} \leq C_d$ with a probability at least $1-\delta$. Here, $C_d$ was defined in~\cref{eqn:cd}. Plug into~\cref{thm:u} to obtain the desired bounds.
\end{proof}


\subsection{Proofs of~\cref{sec:aux} Theorems}

\begin{proof}[\cref{thm:lowrank_schatten}]  

\textbf{Proof of~\cref{eqn:nnorm}}:
Using the unitary invariance of the norms
\[ \uninorm{(I-\proj{Q})A} = \uninorm{(I-\proj{U^*Q})\Sigma}  = \uninorm{(\Sigma^\top(I-\proj{U^*Q})\Sigma)^{1/2}}.\]
 We  use~\cref{eqn:projz} combined with \cref{lemma:schatten} to obtain 
 $$\uninorm{(\Sigma^\top(I-\proj{U^*Q})\Sigma)^{1/2}} \leq \uninorm{(\Sigma^\top(I-\proj{Z})\Sigma)^{1/2}}.$$
With $M_1 \equiv I-(I+F^*F)^{-1}$ and $M_2 \equiv I - F(I+F^*F)^{-1}F^*$, then $\Sigma^\top(I-\proj{Z})\Sigma$ simplifies as
\begin{equation}\label{eqn:block} \Sigma^\top (I-\proj{Z})\Sigma = \bmat{\Sigma_k M_1\Sigma_k & * \\ * &  \Sigma_\perp^\top M_2\Sigma_\perp}.\end{equation}
The square root function is  concave on $[0,\infty)$ and $\Sigma^\top(I-\proj{Z})\Sigma$ is positive semidefinite. Therefore, an extension to Rotfel'd's theorem says~\cite[Theorem 2.1]{lee2011extension} 
\[ \uninorm{(\Sigma^\top(I-\proj{Z})\Sigma)^{1/2}} \leq \uninorm{(\Sigma_k^\top M_1\Sigma_k)^{1/2}} +  \uninorm{(\Sigma_\perp^\top M_2\Sigma_\perp)^{1/2}}. \] 
{
Use the inequalities $M_1 \preceq F^*F$ and $M_2\preceq I$, along with \cref{lemma:schatten} gives 
\begin{equation}
\begin{split}
 \uninorm{(I-\proj{Q})A} \leq & \>  \uninorm{(\Sigma_kF^*F\Sigma_k)^{1/2}} + \uninorm{(\Sigma_\perp\Sigma_\perp)^{1/2}}\\  
\leq & \> \uninorm{F\Sigma_k} + \uninorm{\Sigma_\perp}.
\end{split}
\end{equation}
Use $F\Sigma_k = (\Sigma_\perp\Sigma_\perp)^q\Sigma_\perp \Omega_2\Omega_1^\dagger \Sigma_{k}^{-2q}$ and the sub-multiplicativity to obtain the advertised bounds.}

\textbf{Proof of~\cref{eqn:nnorm_rankr}}: The proof for~\cref{eqn:nnorm_rankr} is similar and is omitted. The main observation is that $A_k$ has only $k$
nonzero singular values.

\textbf{Proof of~\cref{eqn:nnorm_rankr2}}: We follow the strategy in~\cite[Section 3.3]{drineas2018low}. {Recall that $B_k$ is the best rank-$k$ approximation to $B = Q^*A$. With the notation in~\cref{alg:truncation}, note that 
$$QB_k = Q\Uh_{B,k} \Uh_{B,k}^*B = \Uh_k \Uh_{k}^*A = \proj{\Uh_k}A,$$ the triangle inequality gives   
\[ \uninorm{(I-\proj{\Uh_k})A} \leq \uninorm{(I-\proj{\Uh_k})A_k} + \uninorm{(I-\proj{\Uh_k})A_{\perp}}.  \]
}
Since $A_k = \proj{U_k}A_k$, applying strong sub-multiplicativity 
\[ \uninorm{(I-\proj{\Uh_k})A} \leq \uninorm{(I-\proj{\Uh_k})\proj{U_k}} \normtwo{A_k} + \uninorm{A_{\perp}}.\] 
We recognize that $\uninorm{(I-\proj{\Uh_k})\proj{U_k}} = \uninorm{\sin\angle(U_k,\Uh_k)}$, apply~\cref{thm:sintheta} to complete the proof. 
\end{proof}
\begin{proof}[\cref{thm:lowrank}]
The proof is similar to that of the proof of~\cref{thm:lowrank_schatten}. Consider the term of interest $\schattenp{ (I-QQ^*)A}^2$, which can be simplified to 
\[ \schattenp{ (I-QQ^*)A }^2 = \uninorm{A^*(I-QQ^*) A}_{p/2} = \uninorm{ \Sigma^\top (I - \proj{U^*Q}) \Sigma}_{p/2}.\] 
The first equality holds only for $p\geq 2$, whereas the last equality follows because of the unitary invariance.  As in the proof of~\cref{thm:u}, we have 
 \[ Z = U^*Y\Omega_1^\dagger \Sigma_k^{-(2q+1)}\qquad F = (\Sigma_\perp \Sigma_\perp^\top)^q \Sigma_\perp\Omega_2 \Omega_1^\dagger \Sigma_k^{-(2q+1)}.\]
The use of~\cref{eqn:projz} and~\cref{lemma:schatten} ensures 
 \[\uninorm{ \Sigma^\top (I - \proj{U^*Q}) \Sigma}_{p/2} \leq \uninorm{\Sigma^\top (I - \proj{Z}) \Sigma}_{p/2}.\]
 {We apply~\cite[Theorem 2.1]{lee2011extension} to~\cref{eqn:block} with $f(t) =t$ to obtain
 \begin{equation*}
\begin{aligned}\uninorm{\Sigma^\top (I - \proj{Z}) \Sigma}_{p/2} \leq & \> \uninorm{\Sigma_k M_1 \Sigma_k}_{p/2} +\uninorm{\Sigma_\perp^\top M_2  \Sigma_\perp}_{p/2}\\
\leq & \>\uninorm{\Sigma_k F^*F \Sigma_k}_{p/2} + \uninorm{\Sigma_\perp^\top  \Sigma_\perp}_{p/2}\\
= & \> \schattenp{F\Sigma_k}^2 + \schattenp{\Sigma_\perp}^2. \end{aligned} \end{equation*}
We have used $M_1 \preceq F^*F$ and $M_2 \preceq I$.} The rest of the proof is similar to that of~\cref{thm:lowrank_schatten}. 
\end{proof}


\begin{proof}[\cref{thm:sigma}]
{The proof makes heavy use of the partial ordering which was reviewed in the start of \cref{sec:proofs}.} From the inequality $I \succeq QQ^*$, the conjugation rule gives $$A^*A \succeq A^*QQ^*A.$$ Then, Weyl's theorem implies $\lambda_j(A^*A) \geq \lambda_j(A^*QQ^*A)$ for $j=1,\dots,k$. Relating the eigenvalues to the singular values proves the first inequality.

For the second inequality consider again $A^*QQ^*A$. With the aid of~\cref{eqn:projz} 
\begin{equation}\label{eqn:inters} A^*QQ^*A = V\Sigma^\top\proj{U^*Q}\Sigma V^* \succeq V\Sigma^\top\proj{Z}\Sigma V^*. \end{equation}
Therefore, $\lambda_j(A^*QQ^*A) \geq \lambda_j(V\Sigma^\top\proj{Z}\Sigma V^*)$ for $j =1,\dots,k$. Since $V\Sigma^\top\proj{Z}\Sigma V^*$ and $\Sigma^\top\proj{Z}\Sigma$ are similar, they share the same eigenvalues. It can be readily shown that 
\[ \Sigma^\top\proj{Z}\Sigma = \bmat{ \Sigma_k (I+F^*F)^{-1}\Sigma_k & * \\ * & *} .\]
For $j=1,\dots,k$, the eigenvalues of $A^*Q^*QA$ satisfy 
\begin{equation}\label{eqn:ordering} \lambda_j(A^*QQ^*A)  \geq  \lambda_j(V\Sigma^\top\proj{Z}\Sigma V^*) \geq  \lambda_j(\Sigma_k (I+F^*F)^{-1}\Sigma_k).
\end{equation}
{The second inequality follows from the Cauchy interlacing theorem~\cite[Section 10-1]{Par80}.  Applying the properties of partial ordering, we obtain
\[ F^*F \preceq \sigma_{k+1}^{4q+2} \normtwo{\Omega_2\Omega_1^\dagger}^2 \Sigma_k^{-(4q+2)} = \normtwo{\Omega_2\Omega_1^\dagger}^2 \Gamma_k^{4q+2}, \]
where $\Gamma_k = \diag(\gamma_1,\dots,\gamma_k)$ is a diagonal matrix with the singular value gaps. Furthermore, 
\[ \Sigma_k (I+F^*F)^{-1}\Sigma_k \succeq \Sigma_k(I + \normtwo{\Omega_2\Omega_1^\dagger}^2\Gamma_k^{4q+2})^{-1}\Sigma_k.\]
Since the diagonal matrix on the right hand side has its singular values on the diagonals; this fact, combined with~\cref{eqn:ordering} gives for $j=1,\dots,k$
 \[  \sigma_j^2(Q^*A) = \lambda_j(A^*QQ^*A)  \geq  \lambda_j(\Sigma_k (I+F^*F)^{-1}\Sigma_k) \geq \frac{\sigma_j^2(A)}{1 + \normtwo{\Omega_2\Omega_1^\dagger}^2 \gamma_{j}^{4q+2}}.\]
 Taking square-roots, we obtain the desired result.
}
\end{proof}

\section{Numerical Results}

\subsection{Test matrices}\label{ssec:test} To demonstrate the performance of the bounds, we use the following test matrices
\begin{enumerate}
\item \textbf{Controlled gap} The first set of test matrices $A \in \R^{3000\times 300}$ are constructed using the formula
\[ A = \sum_{j=1}^{r}\frac{\text{gap}}{j} \, x_j y_j^\top +
  \sum_{j=r+1}^{300}\frac{1}{j} \,x_j y_j^\top,\]
where $x_j \in \R^{3000}$ and $y_j \in \R^{300}$ are sparse random vectors with non-negative entries generated using the MATLAB commands \verb|sprand(3000,1,0.025)| and \verb|sprand(300,1,0.025)| respectively. The formula above is not an SVD, since the vectors do not form an orthonormal set. Nonetheless, the singular values decay like $1/j$ and the gap between the singular values between $15$ and $16$ is controlled by the parameter $\text{gap}$. We consider three cases:  
\begin{enumerate}
\item Small gap (GapSmall) $\text{gap} = 1$,
\item Medium gap (GapMedium) $\text{gap} = 2$,
\item Large gap (GapLarge) $\text{gap} = 10$.
\end{enumerate}•

\begin{figure}[!ht]\centering
\includegraphics[scale=0.25]{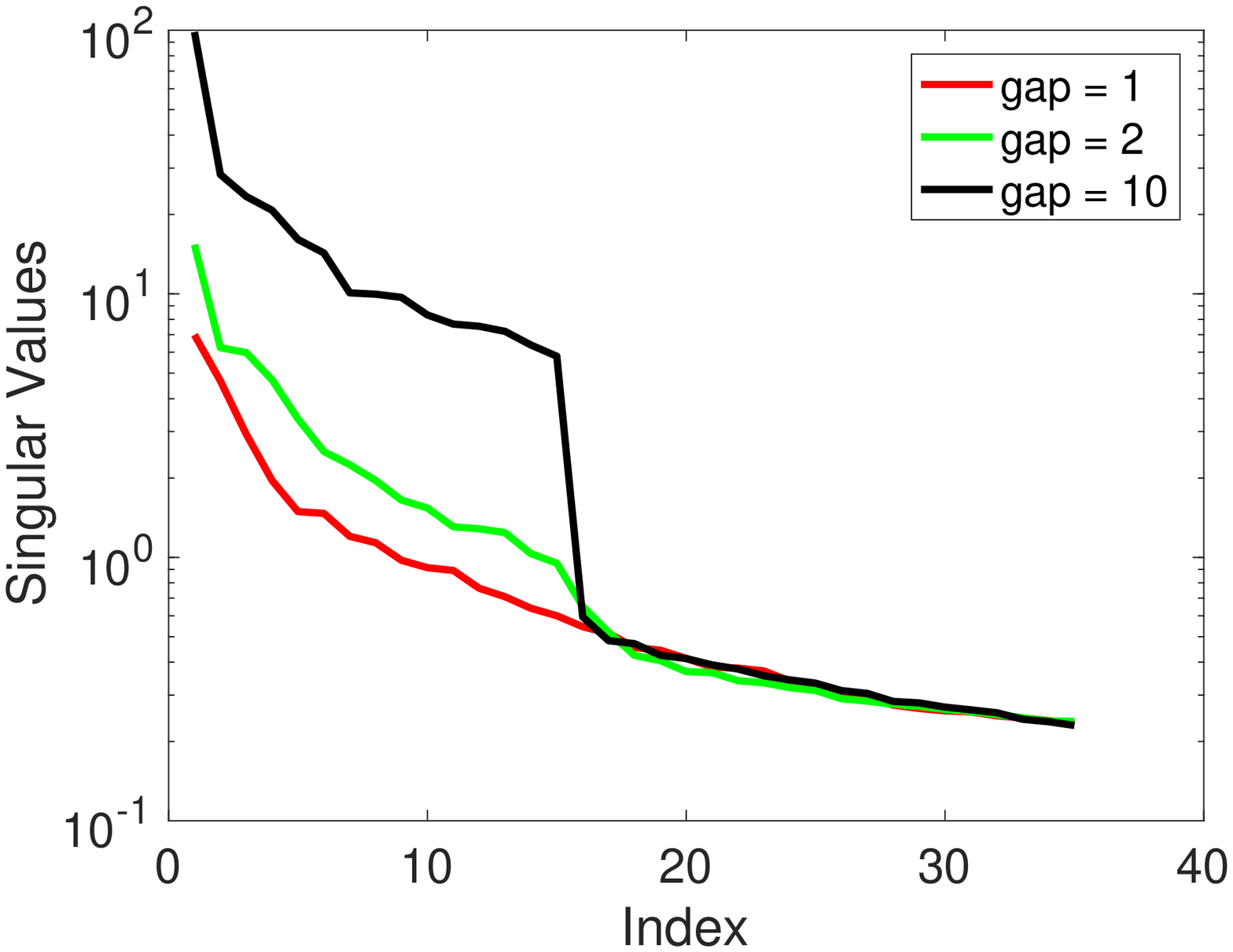}
\includegraphics[scale=0.25]{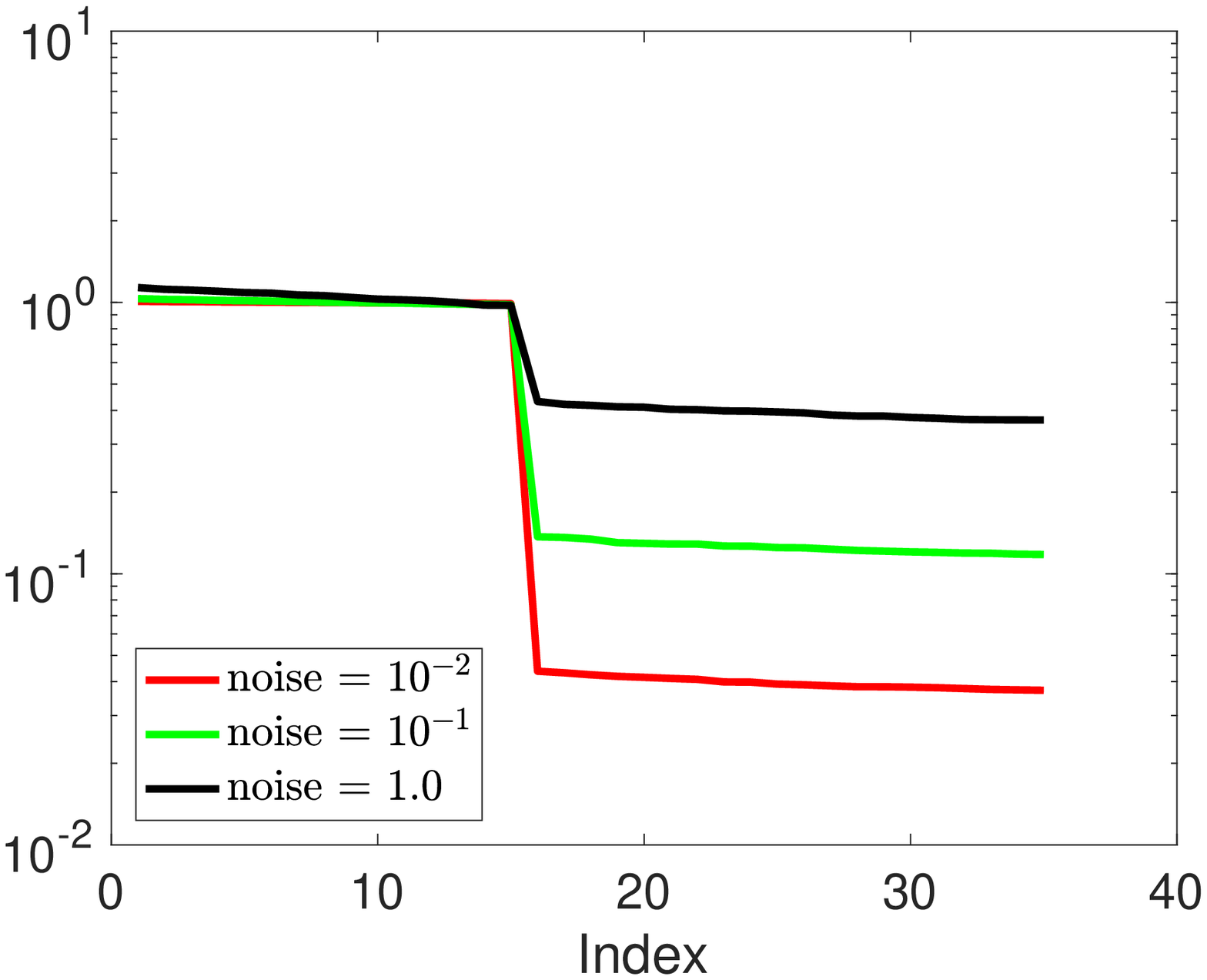}
\includegraphics[scale=0.25]{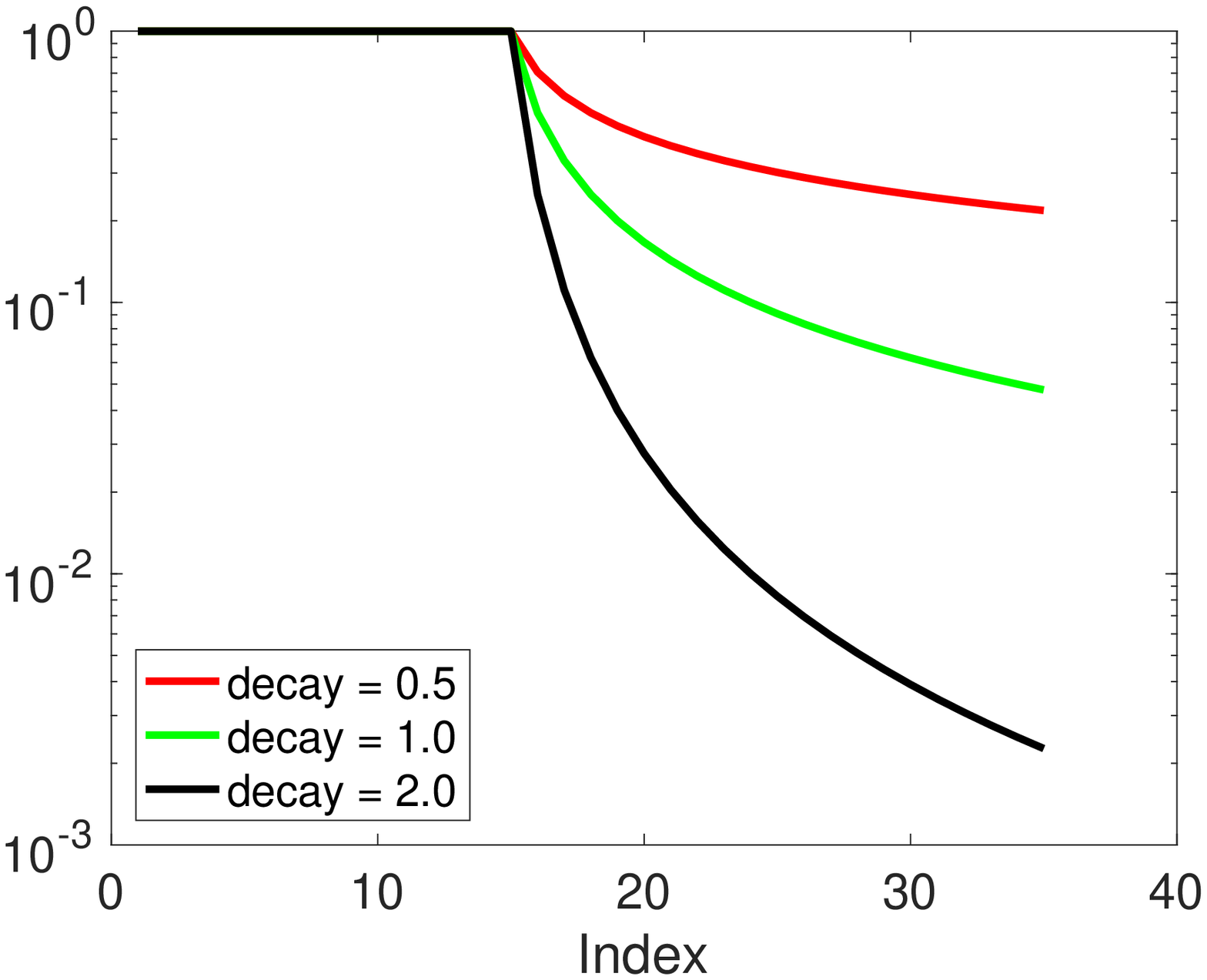}
\caption{Singular value of the matrices from the (left) `Controlled Gap' example, (right) `Low-rank plus noise' example, (below) `Low-rank plus decay' example.  }
\label{fig:svs}
\end{figure}•

\item \textbf{Low-rank plus noise} The matrices are of the form
\[ A = \begin{bmatrix} I_r & 0 \\ 0 & 0\end{bmatrix} + \sqrt{\frac{\gamma_n r}{2n^2}}(G + G^\top),\]
where $G \in \R^{n\times n}$ is a random Gaussian matrix. 
We consider three cases:  
\begin{enumerate}
\item Small noise (NoiseSmall) $\gamma_n = 10^{-2}$,
\item Medium noise (NoiseMedium) $\gamma_n = 10^{-1}$,
\item Large noise (NoiseLarge) $\gamma_n = 1$.
\end{enumerate}•
\item \textbf{Low-rank plus decay} The matrices take the form
\[ A = U\diag(\underbrace{1,1,\dots,1}_r,2^{-d},3^{-d}\dots,(n-r+1)^{-d})V^*.\]
The unitary matrices $U,V$ are obtained by drawing a random Gaussian matrix, and taking its QR factorization. We distinguish between the following cases
\begin{enumerate}
\item Slow decay (DecaySlow): $d = 0.5$,
\item Medium decay (DecayMedium): $d = 1.0$,
\item Fast decay  (DecayFast): $d=2.0$.
\end{enumerate}
\end{enumerate}
The first example is adapted from~\cite{sorensen2016deim}, whereas the second and third examples are drawn from~\cite{tropp2017practical}. In all the examples, the random matrices were fixed by setting the random seed and we the set the parameter $r=15$. The singular values of all the test matrices are plotted in~\cref{fig:svs}. 
\begin{figure}[!ht]%
\centering
\subfigure[GapSmall]{%
\label{fig:cana1}%
\includegraphics[scale=0.24]{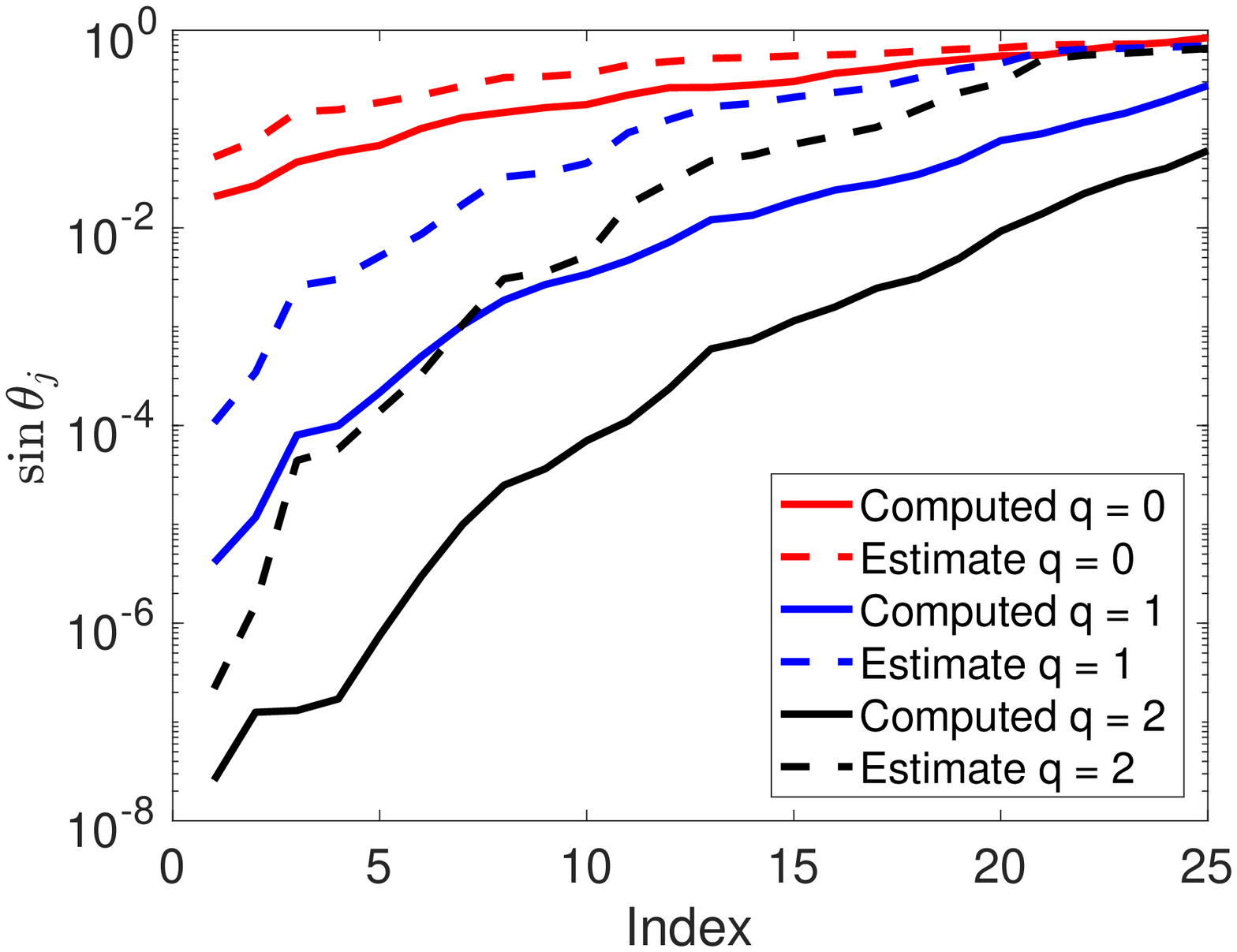}}%
\subfigure[GapMedium]{%
\label{fig:canb1}%
\includegraphics[scale=0.24]{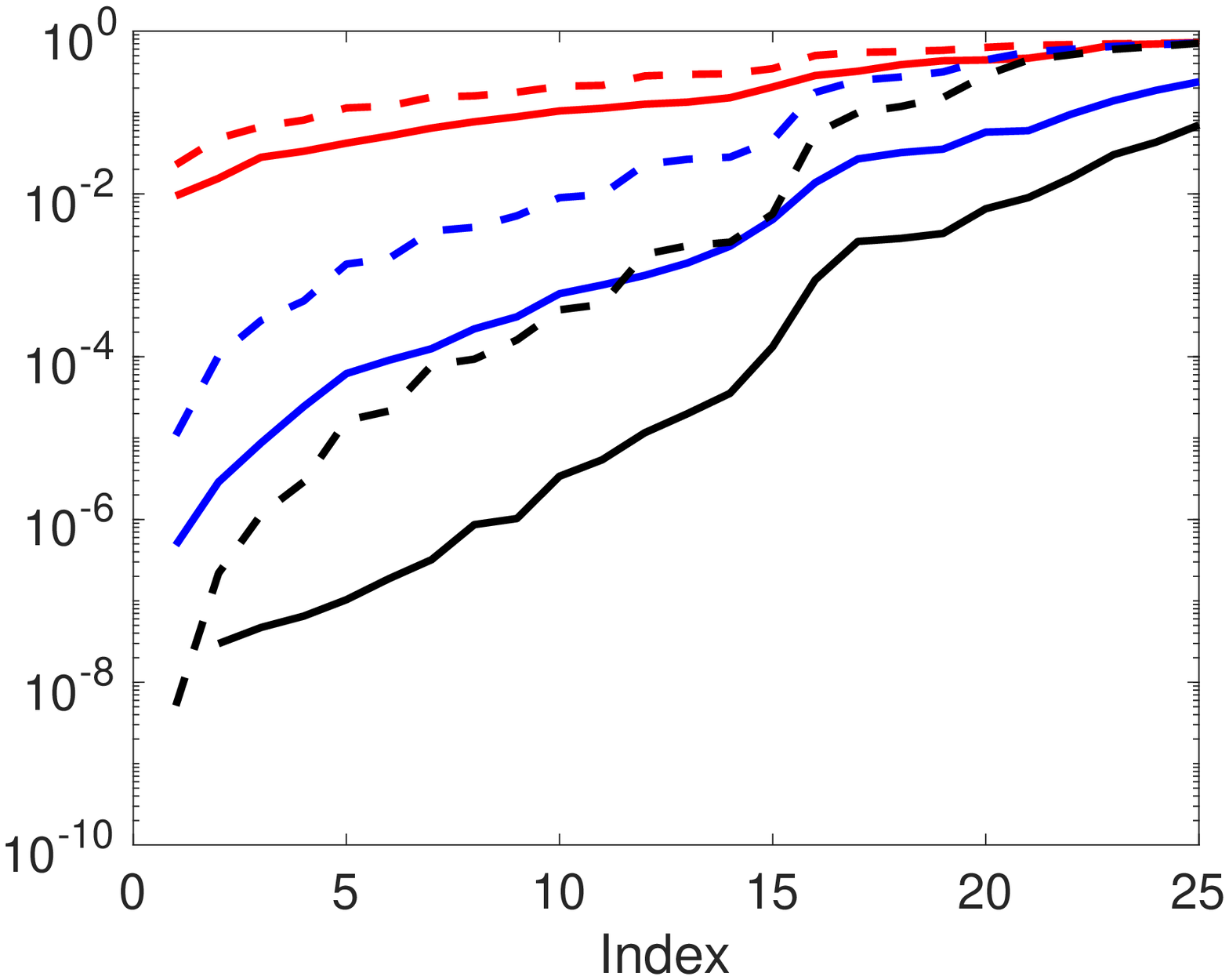}}%
\subfigure[GapLarge]{%
\label{fig:canc1}%
\includegraphics[scale=0.24]{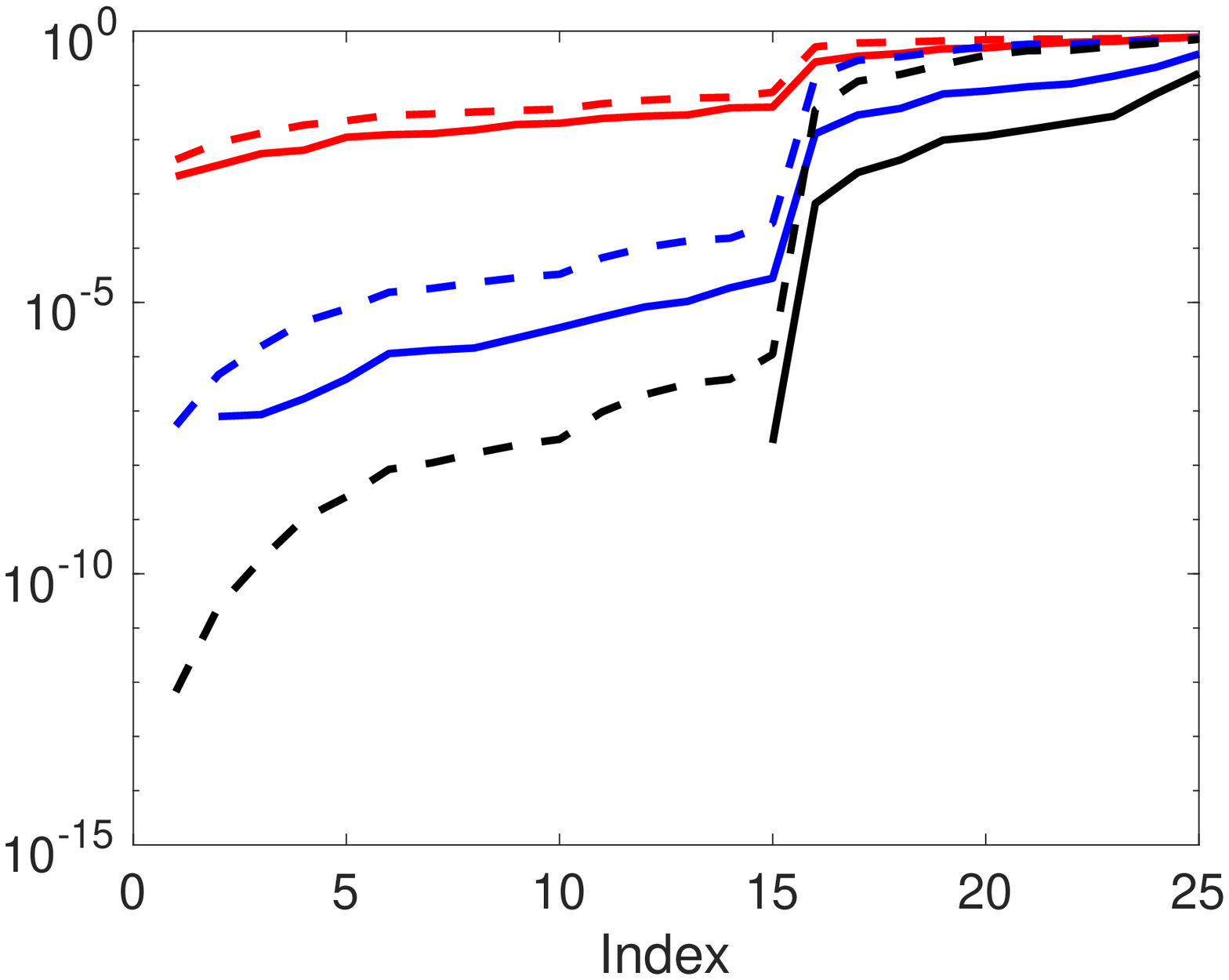}}%
\\
\subfigure[NoiseSmall]{%
\label{fig:cana2}%
\includegraphics[scale=0.24]{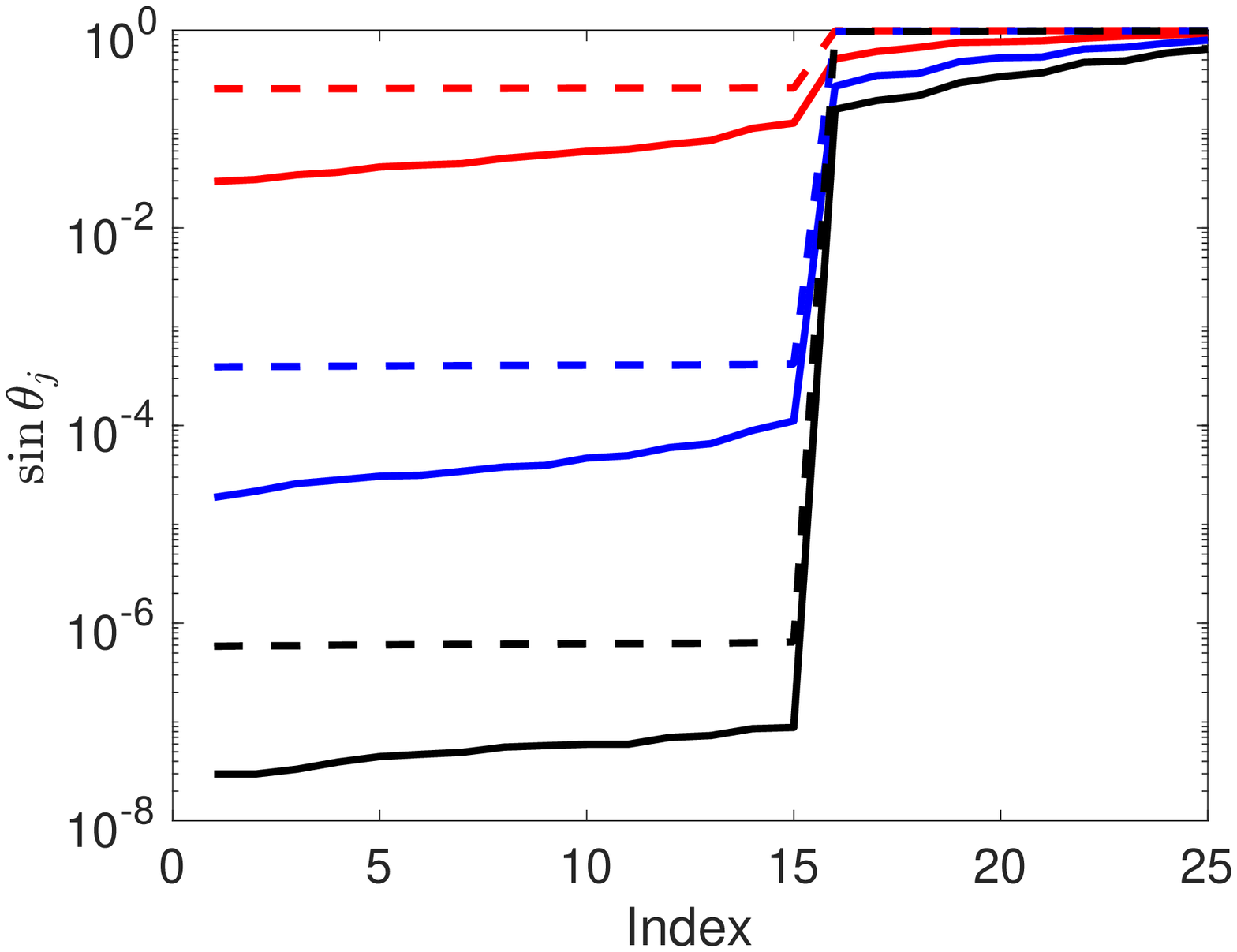}}%
\subfigure[NoiseMedium]{%
\label{fig:canb2}%
\includegraphics[scale=0.24]{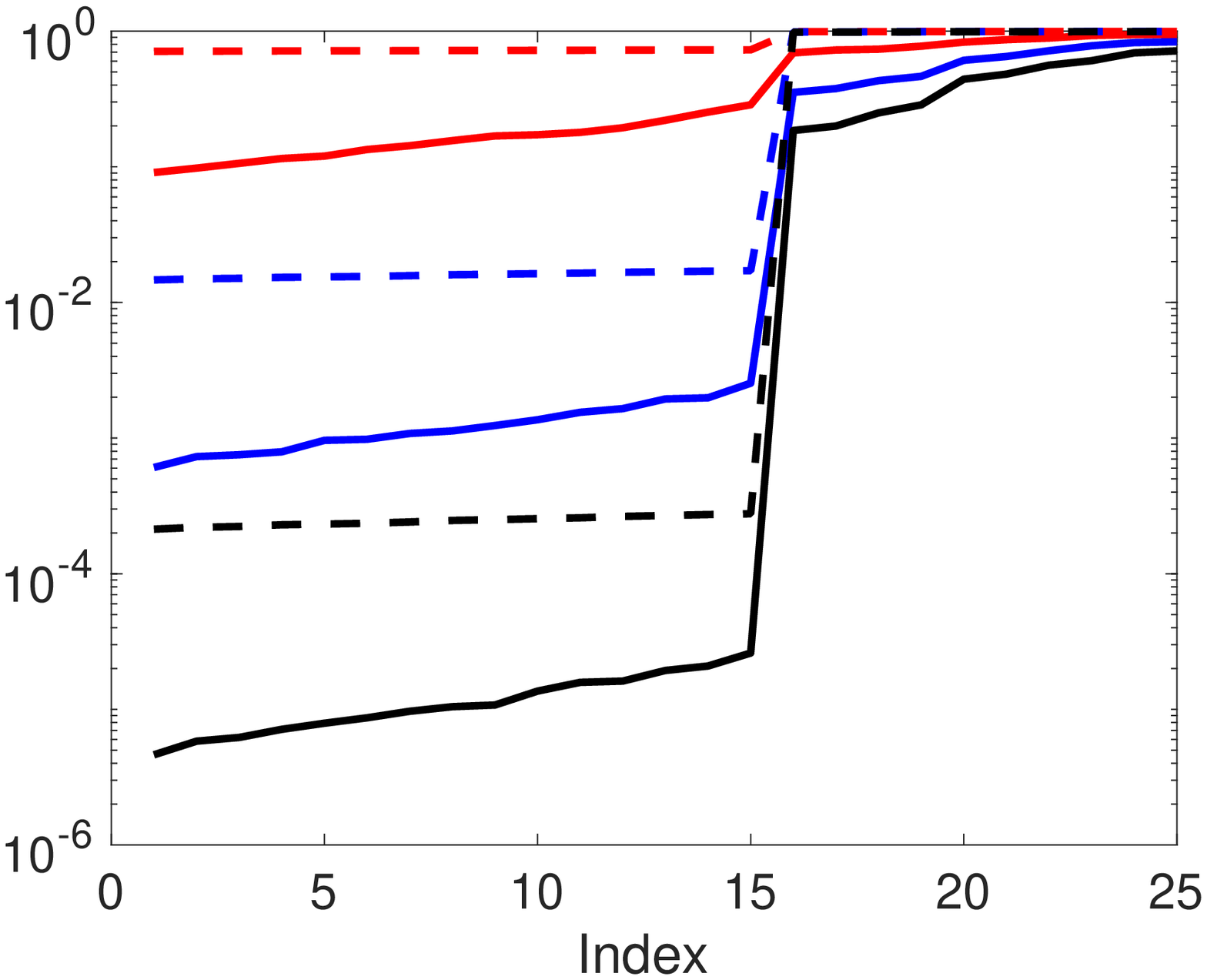}}%
\subfigure[NoiseLarge]{%
\label{fig:canc2}%
\includegraphics[scale=0.24]{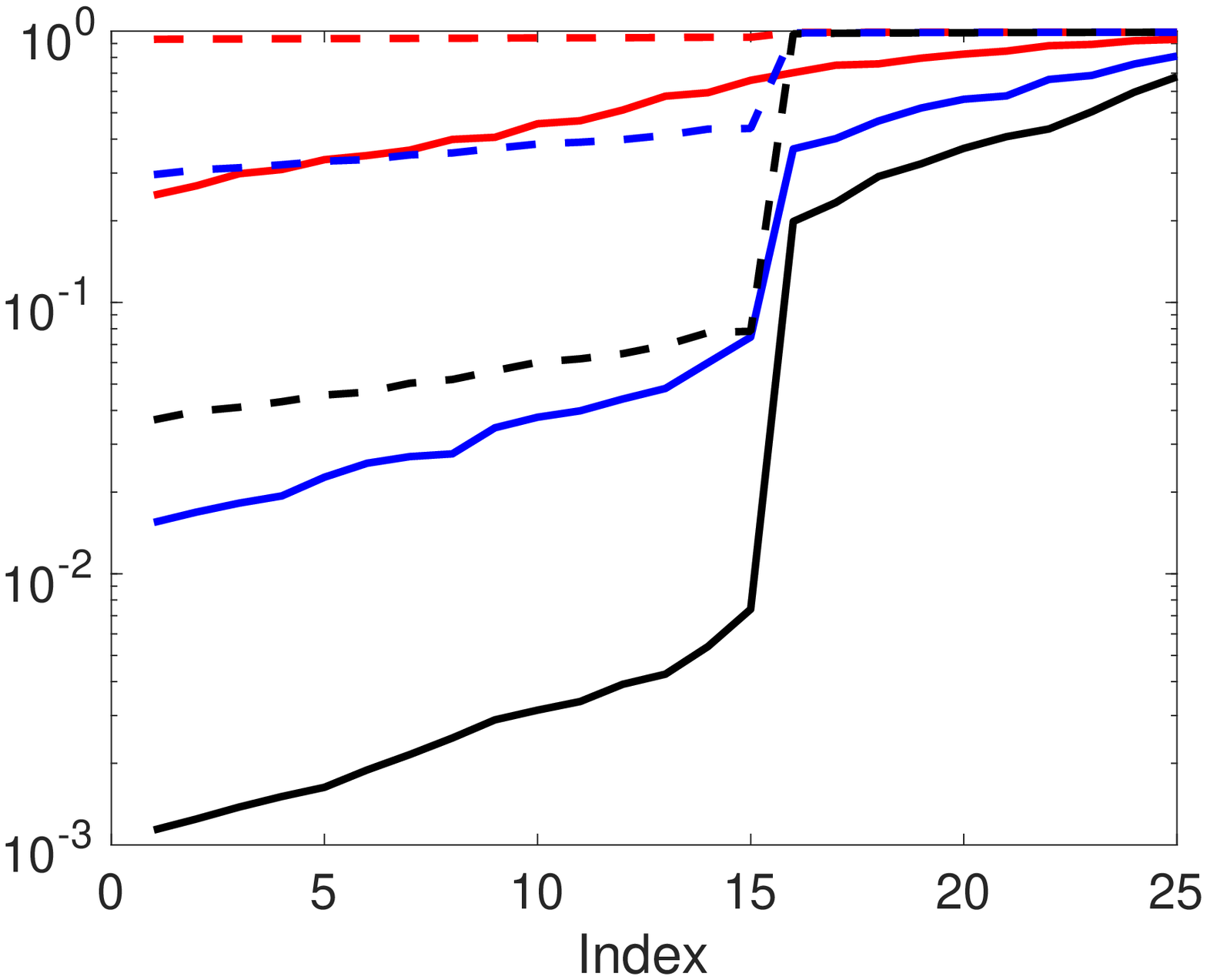}}%
\\
\subfigure[DecaySlow]{%
\label{fig:cana3}%
\includegraphics[scale=0.24]{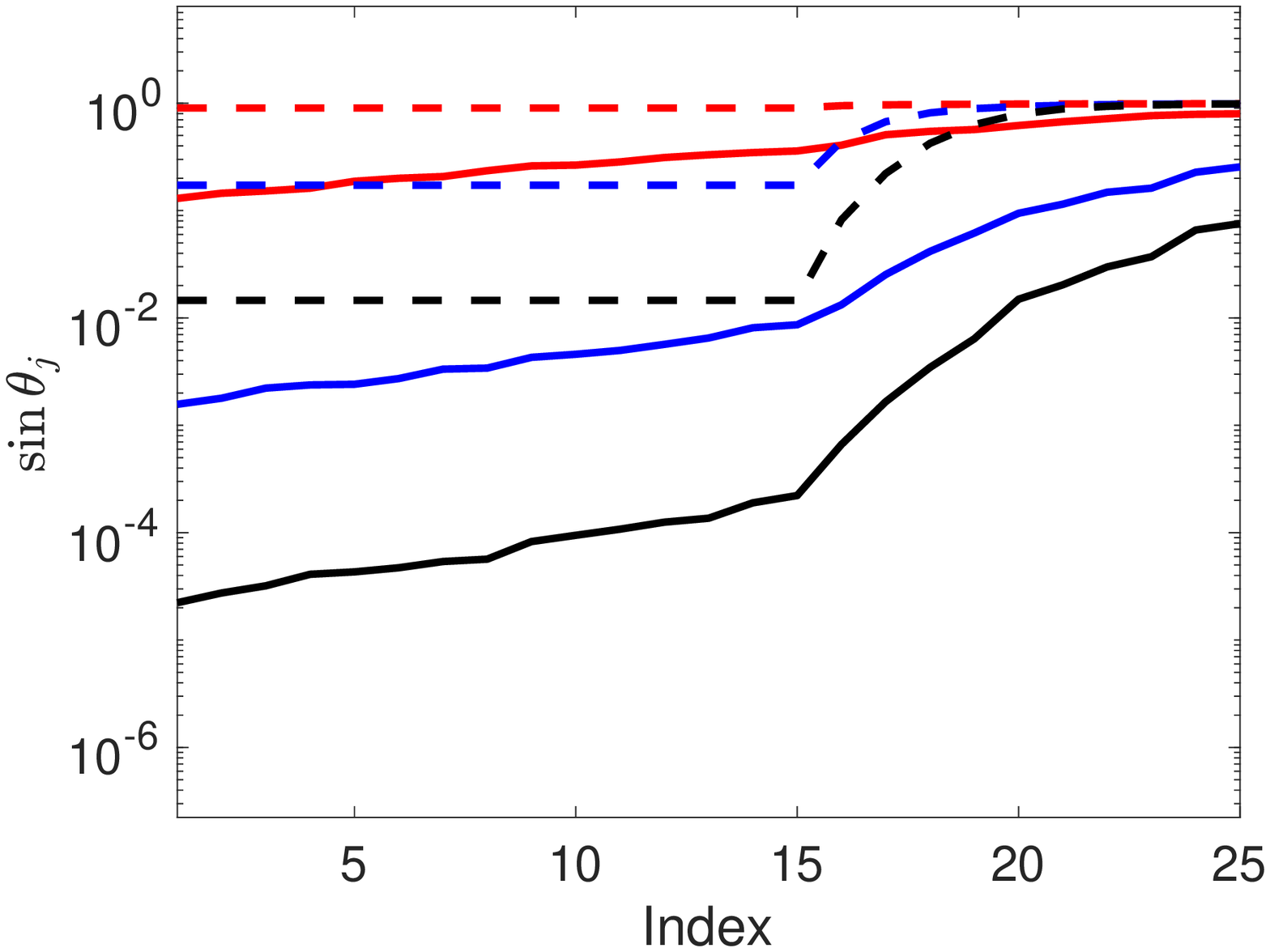}}%
\subfigure[DecayMedium]{%
\label{fig:canb3}%
\includegraphics[scale=0.24]{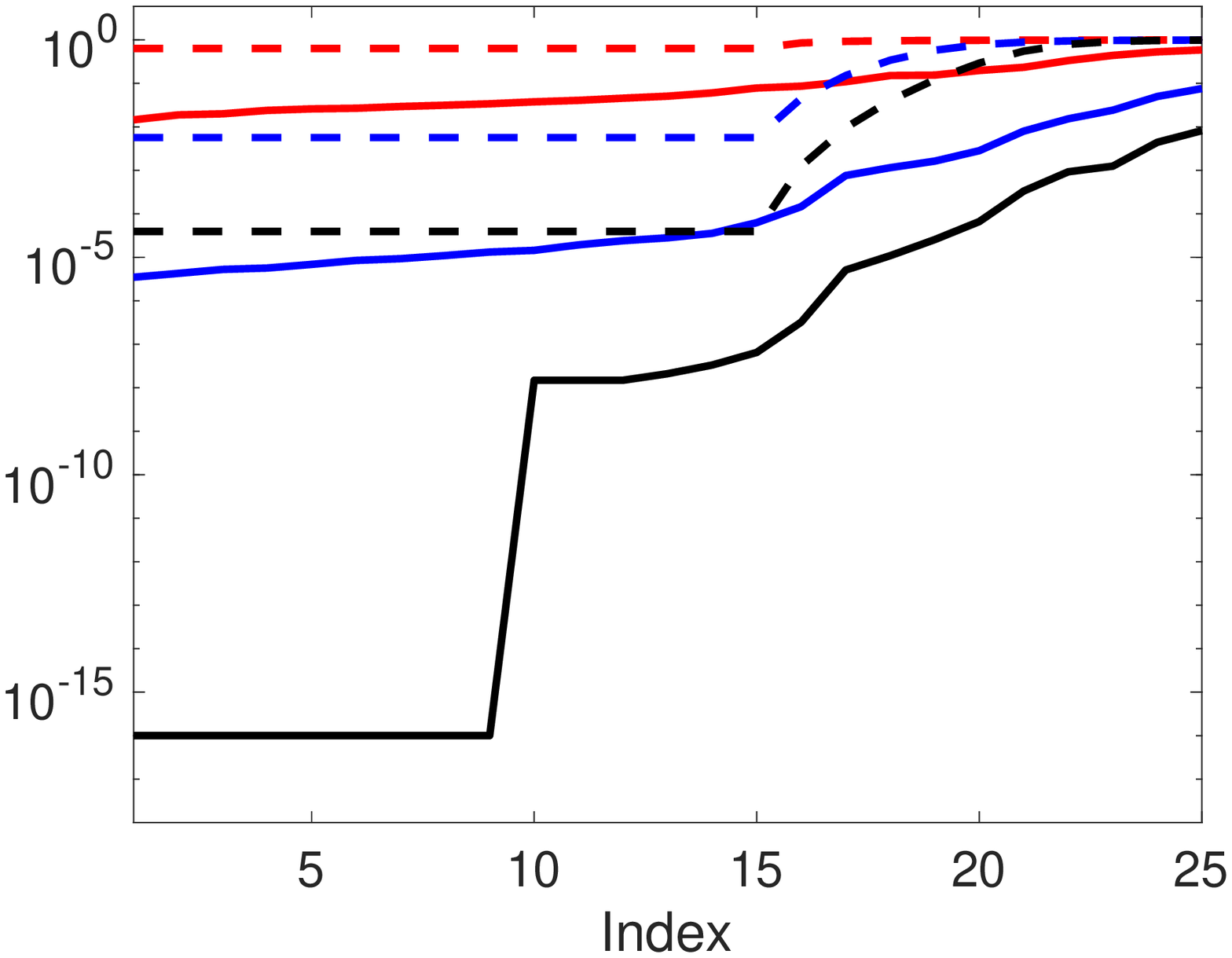}}%
\subfigure[DecayFast]{%
\label{fig:canc3}%
\includegraphics[scale=0.24]{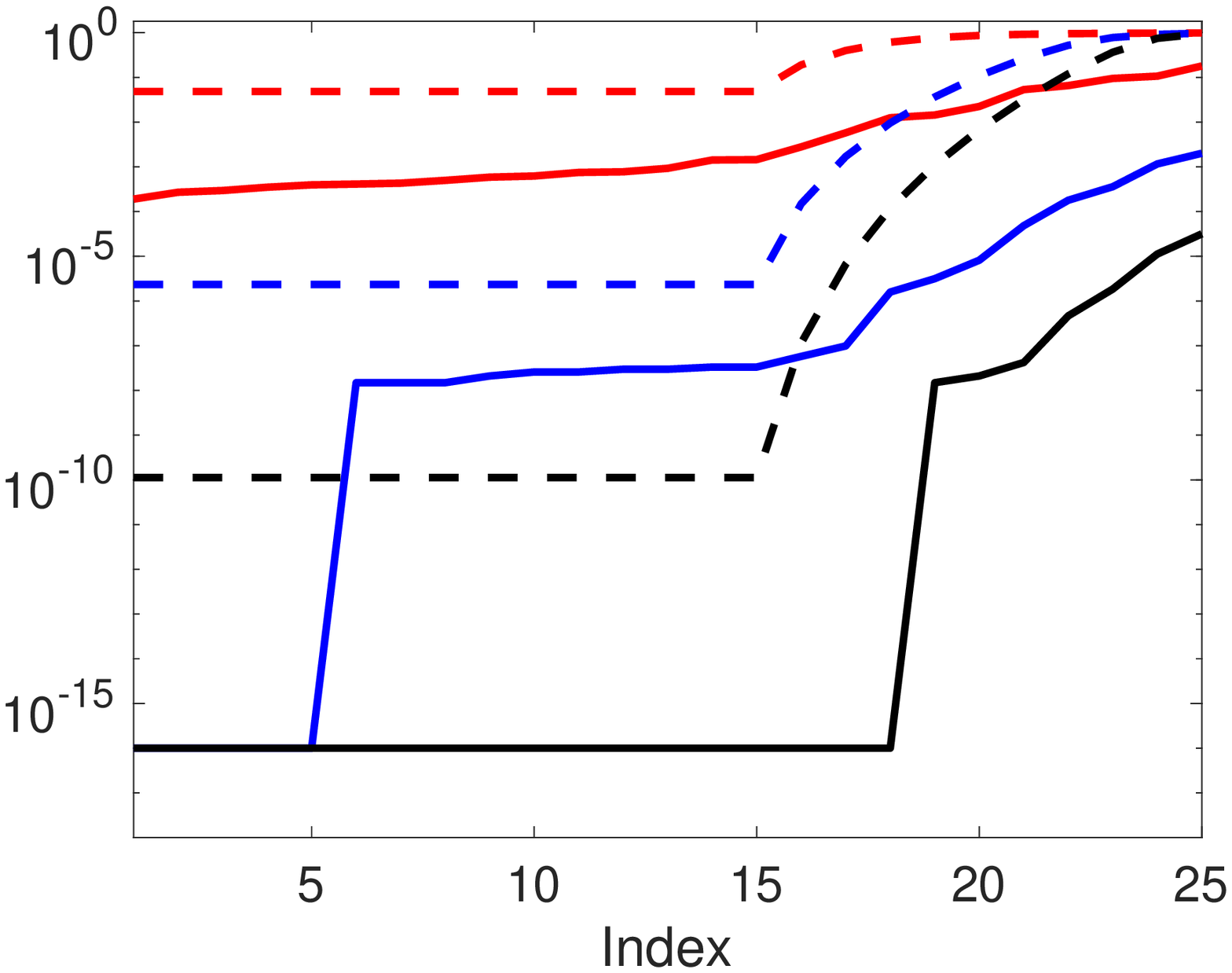}}%
\caption{Plots of $\sin \theta_j $ for $j=1,\dots,k$. The test matrices were described in \cref{ssec:test}. The target rank $k=25$ and an oversampling parameter of $20$ was chosen for all the experiments. The solid lines correspond to the computed values, the dashed lines correspond to bounds obtained using~\cref{thm:u}. The parameter $q$ corresponds to the number of subspace iterations. }
\label{fig:can}
\end{figure}
\subsection{Canonical angles}
For the first numerical example, we use the $9$ test matrices in~\cref{ssec:test}. For each matrix, we chose an oversampling parameter $\rho=20$ and the target rank $k$ was chosen to be $25$. The starting guess $\Omega$ was taken to be a random Gaussian matrix.

\subsubsection{No extraction} We plot the canonical angles $\sin\angle(U_k,\Uh)$ in solid lines, the corresponding bounds from~\cref{thm:u} are also plotted in dashed lines.  The results are displayed in~\cref{fig:can}. We make the following general observations:
{
\begin{itemize}
\item The influence of the subspace iterations on the canonical angles is clear: the angles become smaller as the number of iterations $q$ increases. This implies that the subspace is becoming more accurate. 
\item If there is a large singular value gap in the spectrum, this means that all the canonical angles below that index are captured accurately. This is prominently seen in~\cref{fig:canc1}, in which there is a large gap between singular values $15$ and $16$. Similar observations can be made in the other figures. 
\item As the decay rate of the singular values increases, the corresponding canonical angles become smaller. 
\item  In most figures the bounds are qualitatively informative, but in some figures, the bounds are also quantitatively accurate (e.g., GapLarge).
\item  Similar results were observed for $\sin\angle(V_k,\Vh)$ and, therefore, omitted. 
\end{itemize}•
}
{
We now make observations specific to the test examples: 
\begin{description}
\item [1. Gap examples] The computed canonical angles decrease as the gap increases, and with more iterations. The test matrices (GapMedium and GapLarge) have both a decay in the singular values and a prominent singular value gap between indices $15$ and $16$.   These matrices satisfy the assumptions of our analysis, and therefore the bounds can be expected to be good. We see that as the size of the gap increases, the bounds become more accurate in accordance with Theorem~\ref{thm:u}. GapSmall has decay in the singular values but no special singular value gap. Even in this case, the bounds are qualitatively good.
\item [2. Noise examples]  NoiseSmall is close to a low-rank matrix and  there is a large singular value gap at index $15$. For this example, the bounds are qualitatively good. As the level of noise increases, the gap decreases and therefore, the computed angles increase, as predicted by~\cref{thm:u}. The bounds are uninformative for $q=0$, but qualitatively good for $q=1$ and $2$. Compared to the Gap examples, the bounds are not as sharp since there is very little decay in the singular values. 
\item [3. Decay examples] In these examples, the singular values decay beyond index $15$ but there is no prominent gap. As the rate of decay increases, in general, the canonical angles decrease. It is also seen that the bounds are qualitatively accurate (except for $q=0$). 
\end{description}
 }

\subsubsection{Extraction step} Our next experiment tests the effect of the extraction step on the accuracy of the canonical angles. We now compute $\sin\theta_j'$ and $\sin\nu_j'$ for the test matrices described in~\cref{ssec:test}. We plot the quantities $\max\{\sin\theta_j',\sin\nu_j'\}$ for $j=1,\dots,k$ in solid lines. The corresponding bounds from~\cref{thm:sintheta} are plotted in dashed lines. Here, the target rank was chosen to be $k=15$, to exploit the singular value gap in the matrices. We make the following general observations:
{
\begin{itemize}	
\item The extraction step did not significantly affect the canonical angles and the accuracy is comparable to~\cref{fig:can}.  The subspaces are more accurate as the number of iterations increase, and if there is a large singular value gap at index $j$, then the canonical angles with index $j' < j$  are captured accurately.
\item Although the canonical angles are small, compared to~\cref{thm:u}, the bounds in~\cref{thm:sintheta} are not as accurate. One reason is that the upper bounds in~\cref{thm:u} are at most $1$, but the bounds in~\cref{thm:sintheta} are allowed to be greater than $1$. Furthermore, the bound in~\cref{thm:sintheta} has the factor $1/(1-\gamma_k)$ in the denominator, which can be quite large when there is a small singular value gap. It may be possible to derive better bounds, but we could not immediately see how to derive them. 
\item We also compared the accuracy of the individual singular vectors (not shown here). The results and the conclusions are similar. 
\end{itemize}
}
{
We now make observations specific to the test examples: 
\begin{description}
\item [1. Gap examples] The behavior of the computed canonical angles is very similar to that without the extraction step. In general, the angles decrease as the gap increases. When the parameter $\text{gap}$ is small, the singular value ratio $\gamma_k$ is large, and $(1-\gamma_k)^{-1}$ is small. This explains why the bounds are bad for GapSmall and GapMedium, and show little improvement with more subspace iterations. Only for the GapLarge example with $q=0$, the bounds are qualitatively good. 
\item [2. Noise examples] The computed canonical angles decrease as the noise decreases. In all three examples, the  bounds are qualitatively good. The bounds are better for NoiseSmall and NoiseMedium because the singular value gap between indices $15$ and $16$ is bigger than that for NoiseLarge. 
\item [3. Decay examples] The computed canonical angles become smaller as the decay of the singular values increases. In these examples, there is no prominent gap, so the bounds don't capture the behavior well. However, the computed angles are small, and the subspace is accurate. 
\end{description}
 }

\begin{figure}[!ht]%
\centering
\subfigure[GapSmall]{%
\label{fig:exta1}%
\includegraphics[scale=0.24]{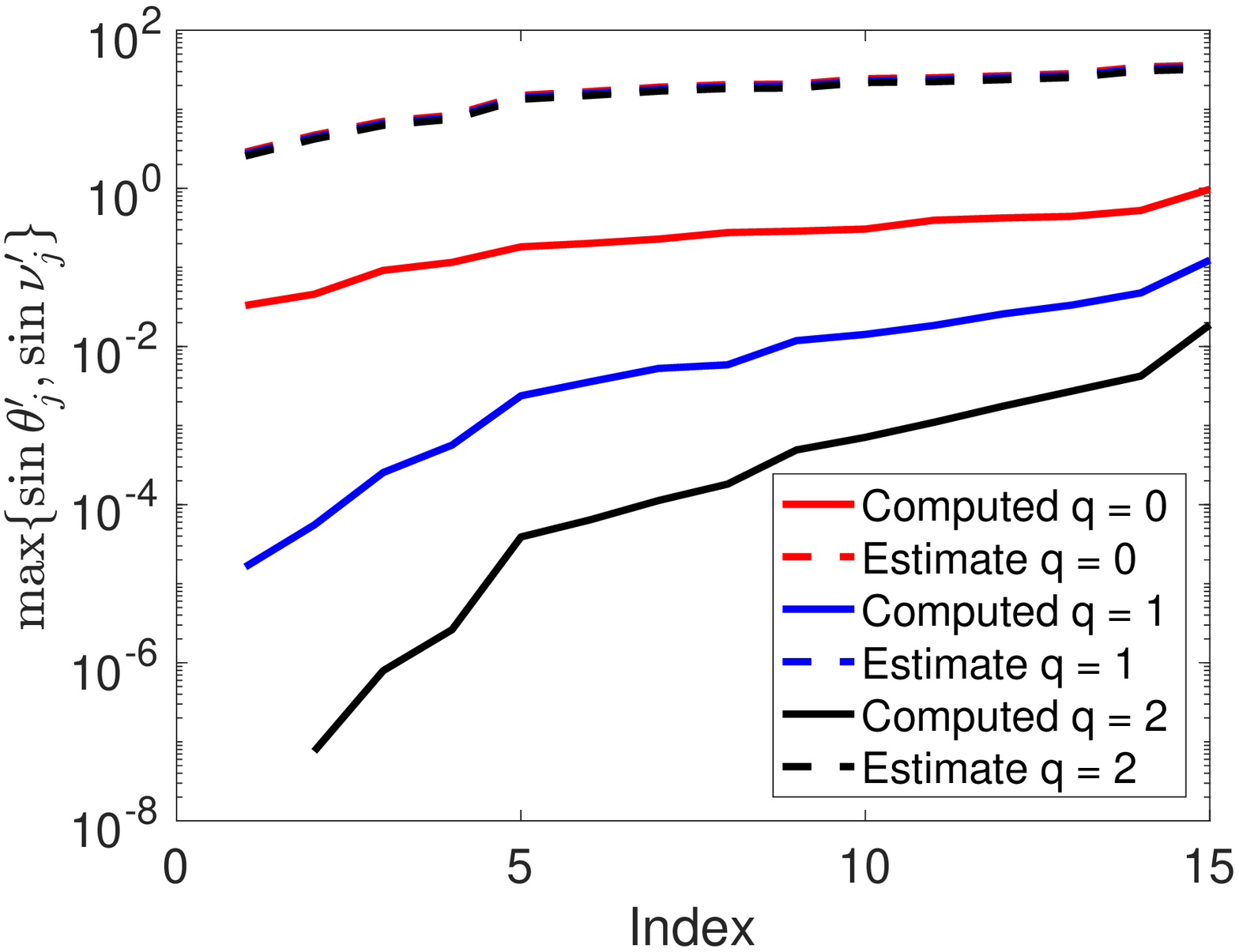}}%
\subfigure[GapMedium]{%
\label{fig:extb1}%
\includegraphics[scale=0.24]{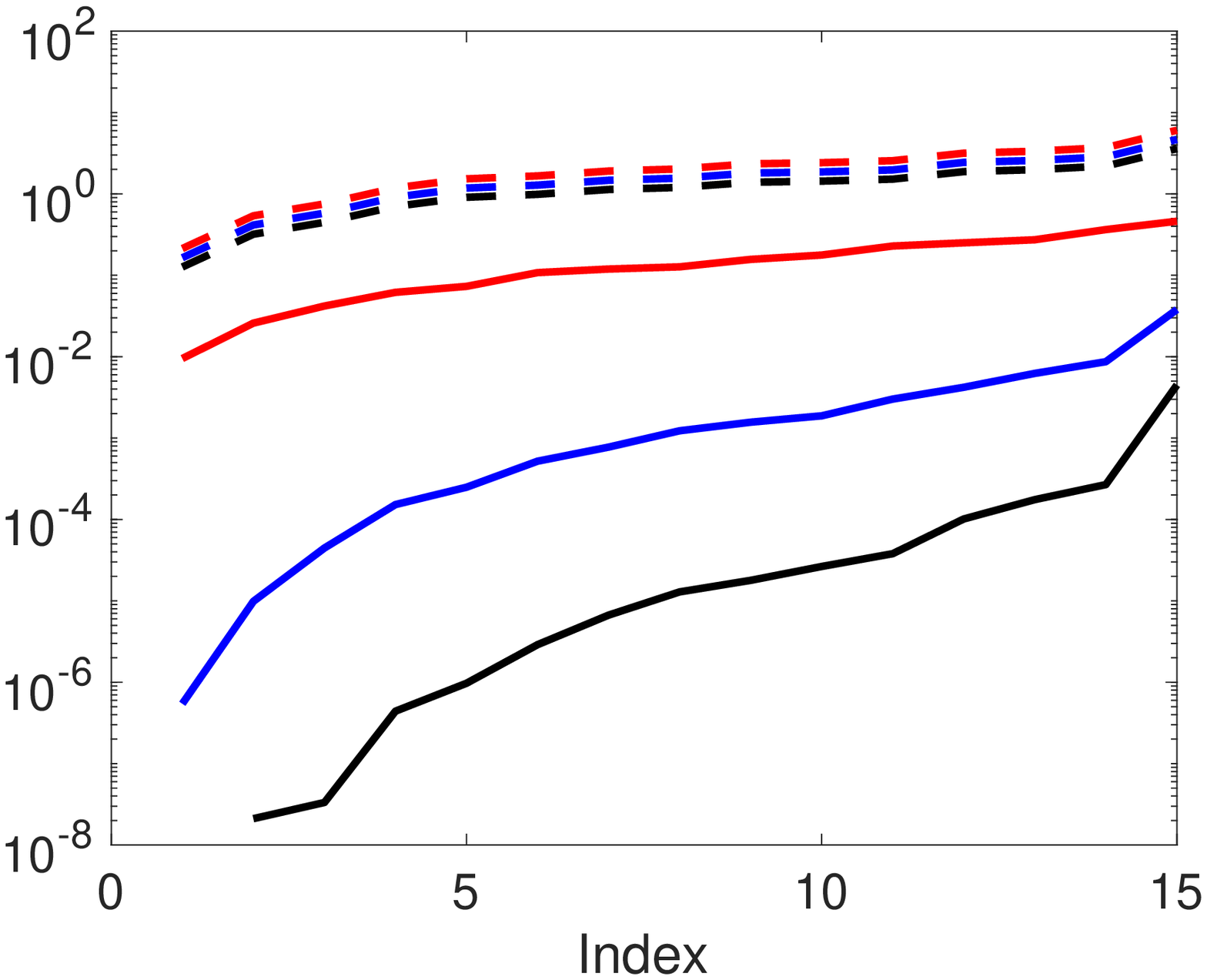}}%
\subfigure[GapLarge]{%
\label{fig:extc1}%
\includegraphics[scale=0.24]{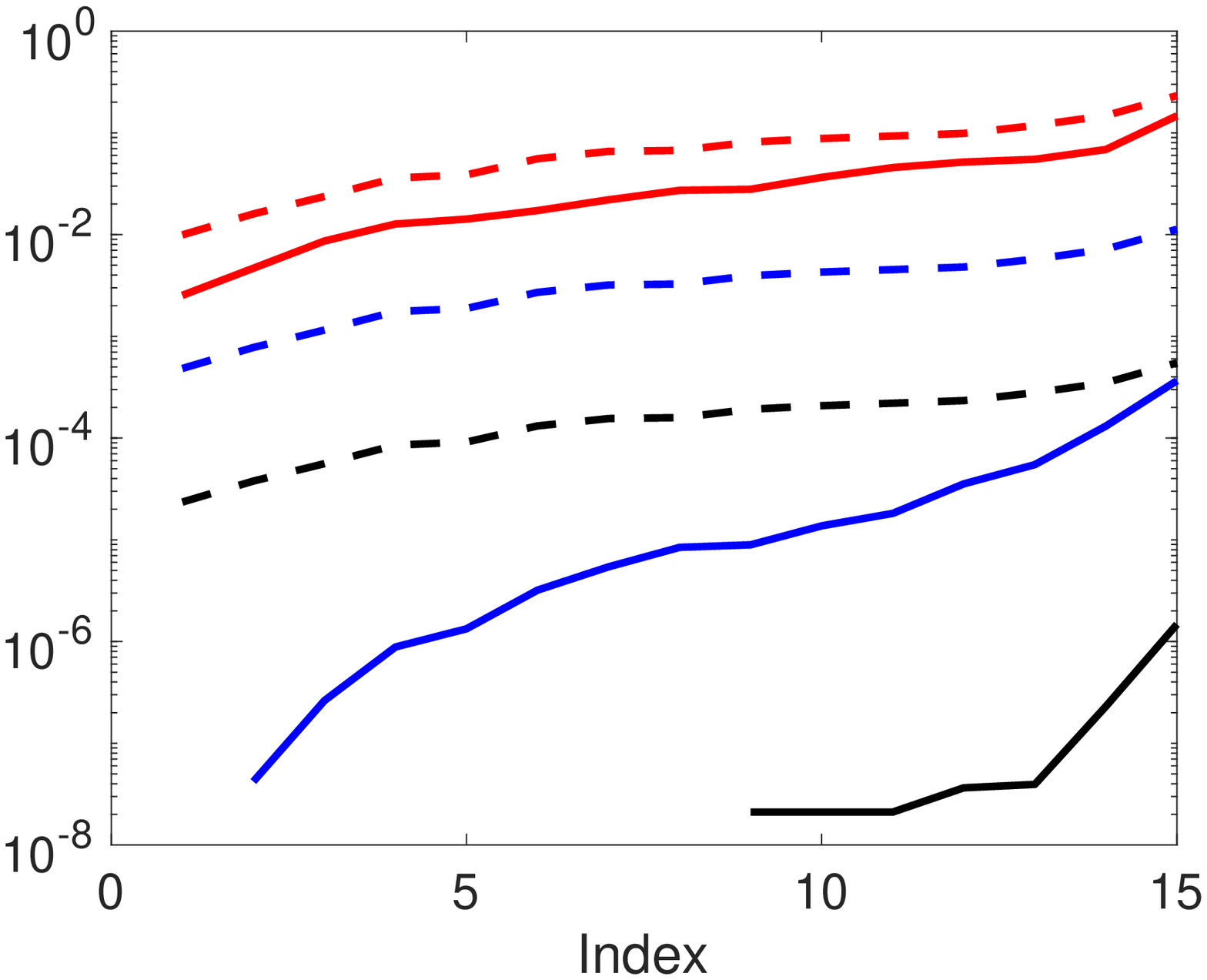}}%
\\
\subfigure[NoiseSmall]{%
\label{fig:exta2}%
\includegraphics[scale=0.24]{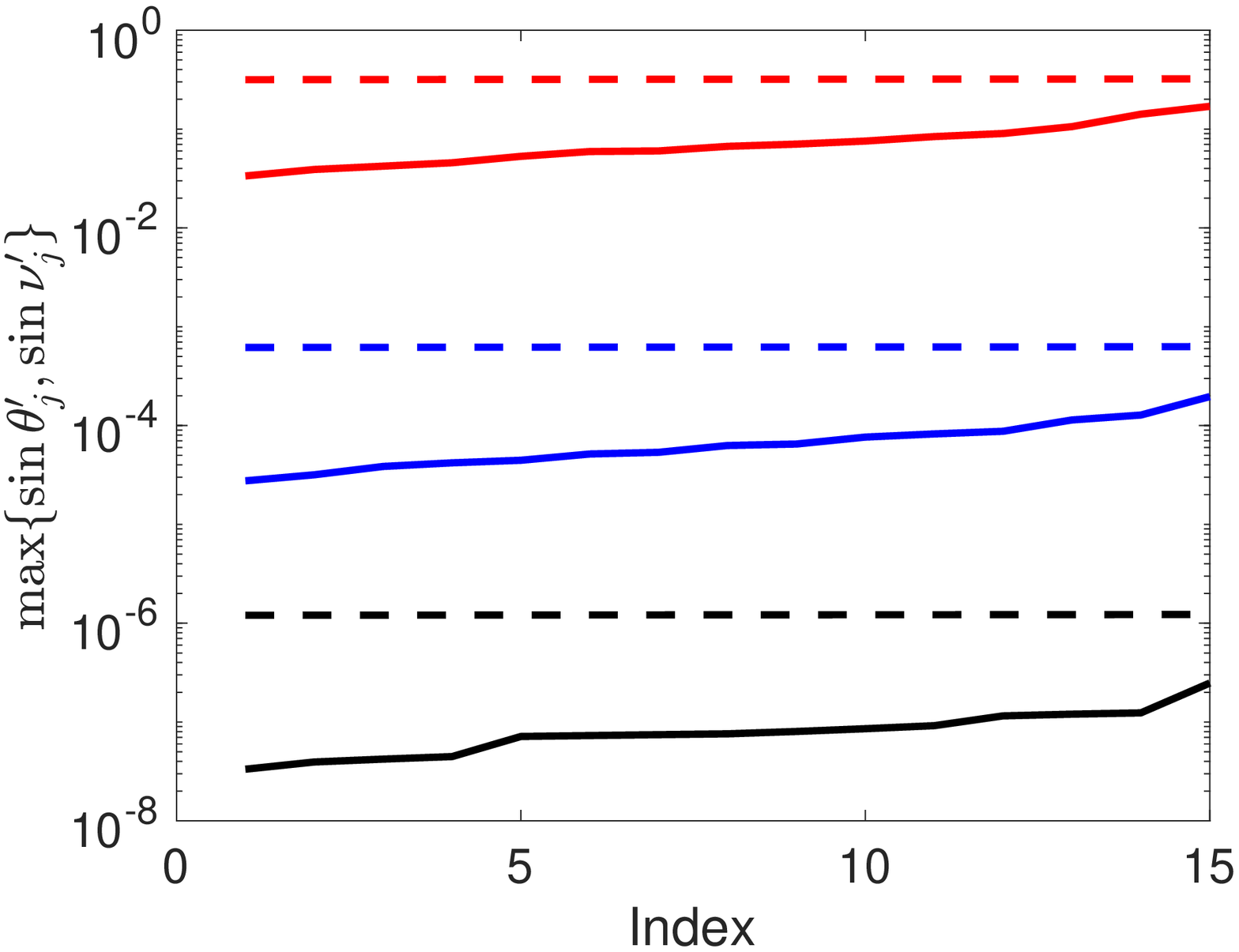}}%
\subfigure[NoiseMedium]{%
\label{fig:extb2}%
\includegraphics[scale=0.24]{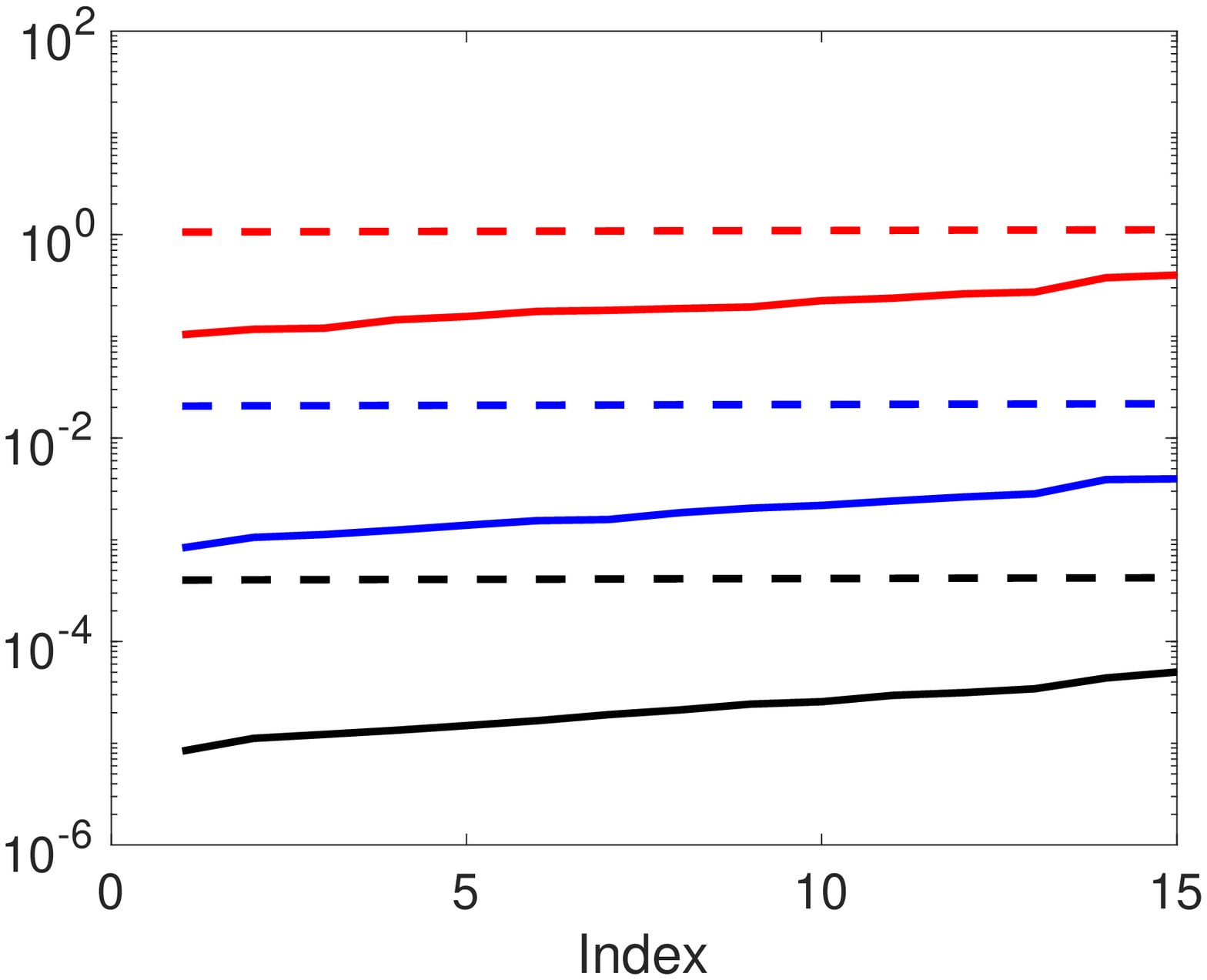}}%
\subfigure[NoiseLarge]{%
\label{fig:extc2}%
\includegraphics[scale=0.24]{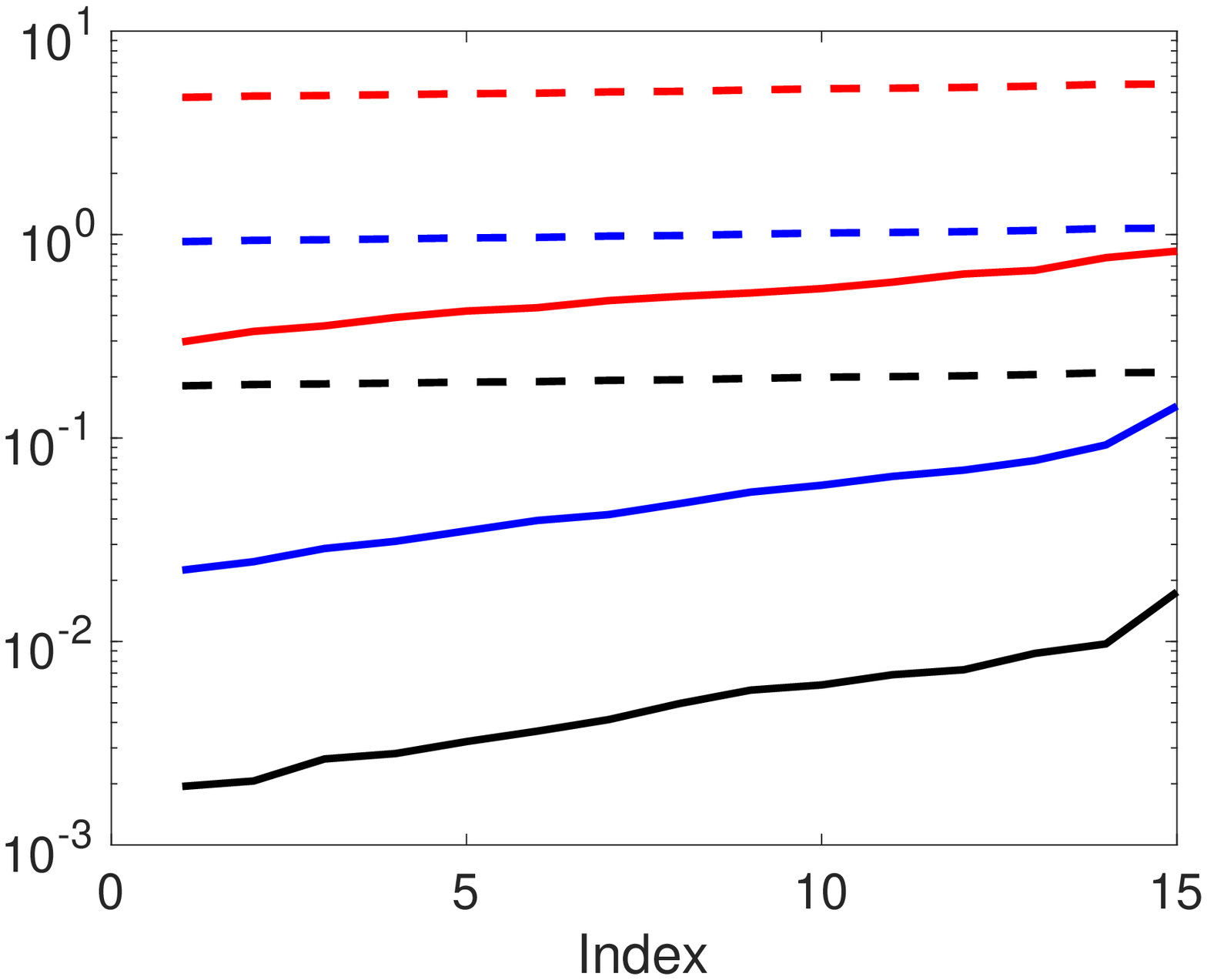}}%
\\
\subfigure[DecaySlow]{%
\label{fig:exta3}%
\includegraphics[scale=0.24]{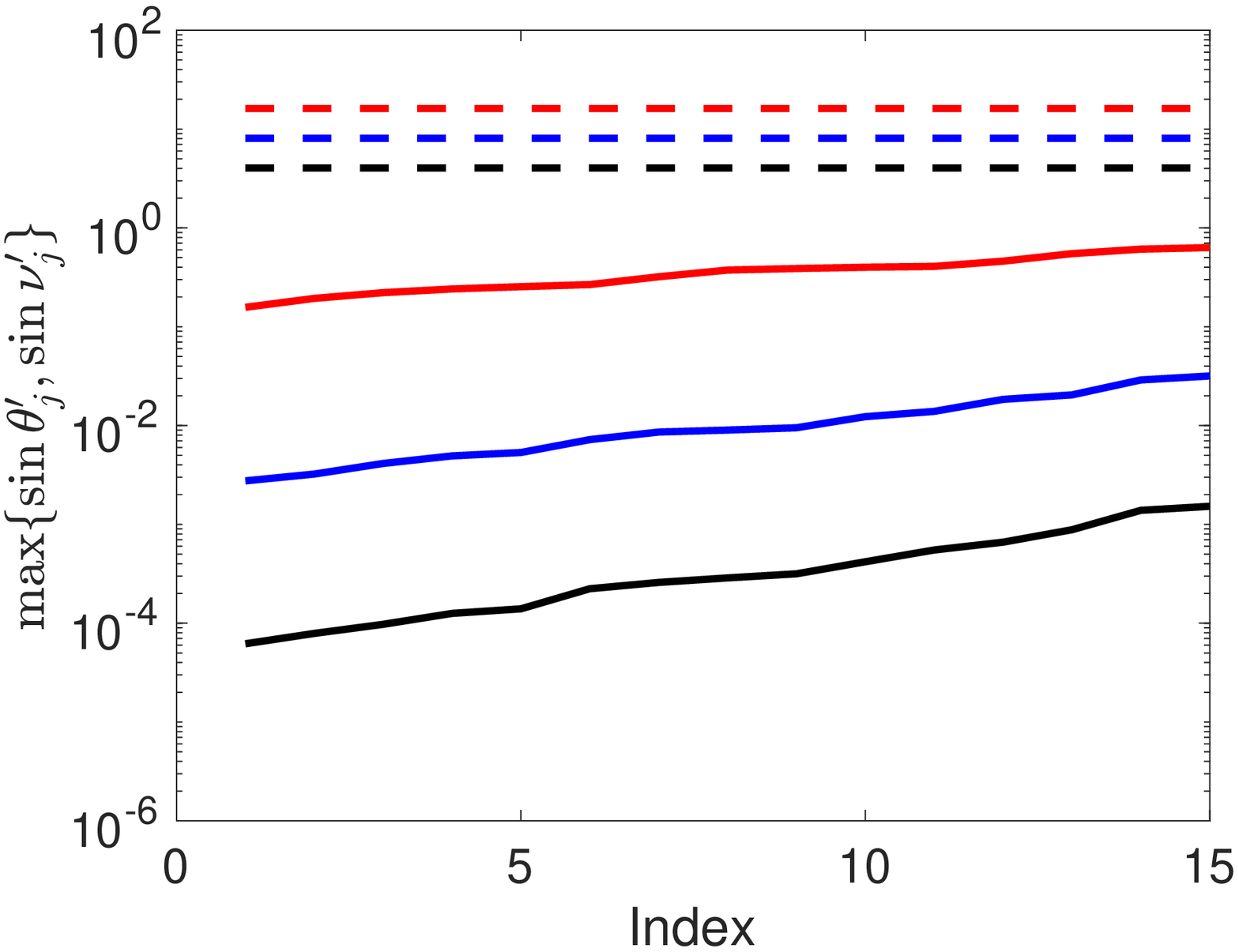}}%
\subfigure[DecayMedium]{%
\label{fig:extb3}%
\includegraphics[scale=0.24]{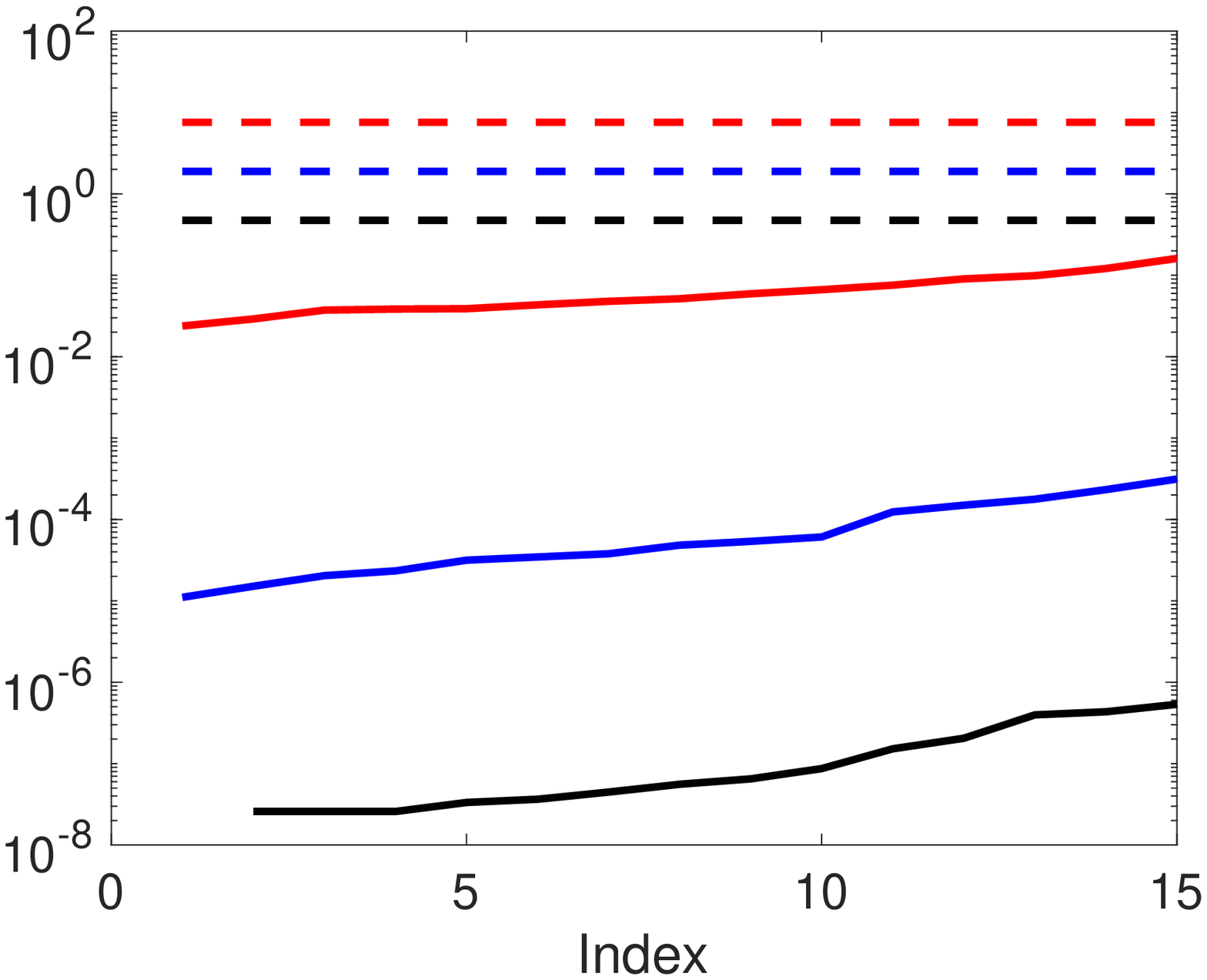}}%
\subfigure[DecayFast]{%
\label{fig:extc3}%
\includegraphics[scale=0.24]{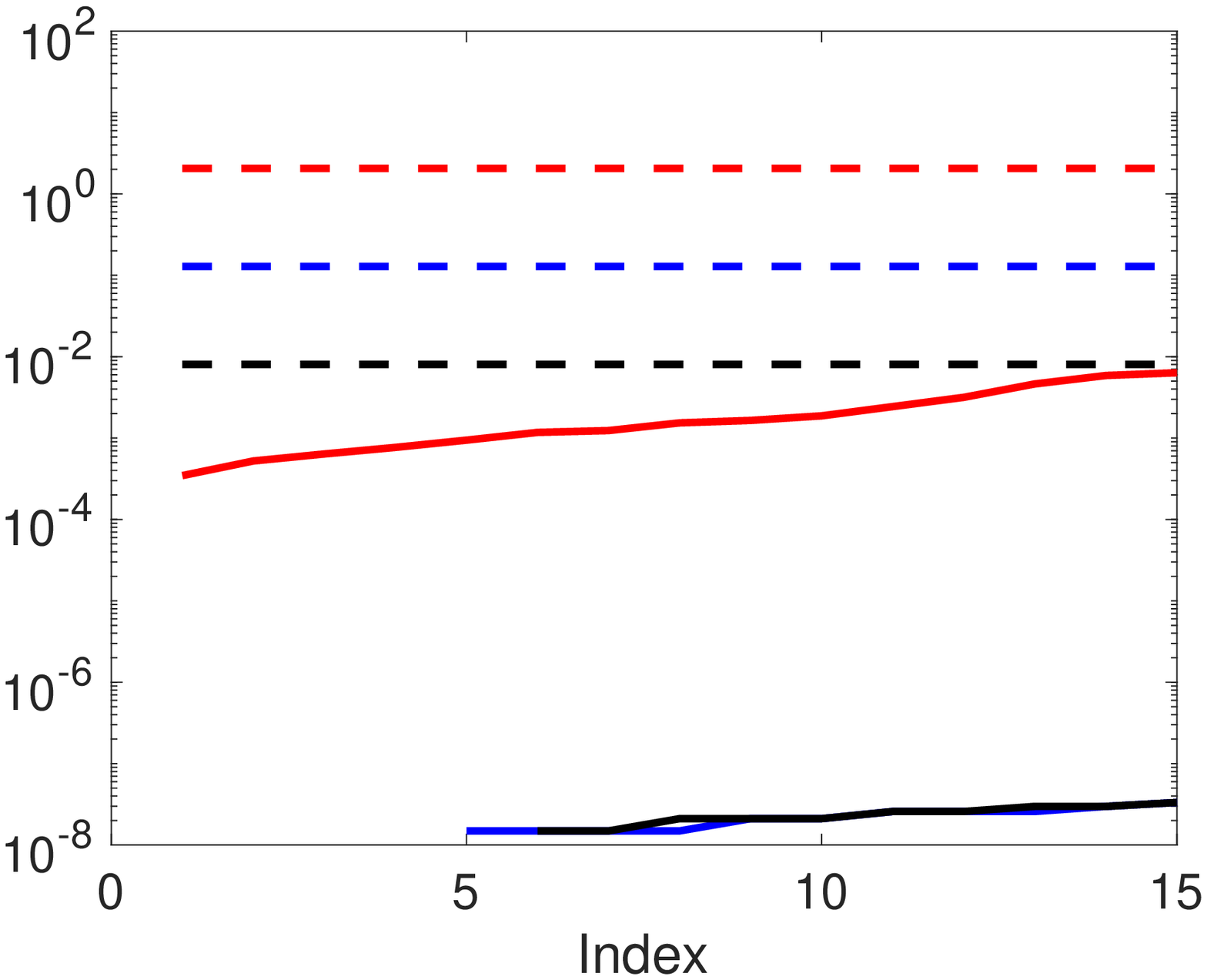}}%
\caption{Plots of $\max\{\sin \theta_j',\sin\nu_j'\} $ for $j=1,\dots,k$. The test matrices were described in~\cref{ssec:test}. The target rank $k=15$ and an oversampling parameter of $20$ was chosen for all the experiments. The solid lines correspond to the computed values, the dashed lines correspond to bounds obtained using~\cref{thm:sintheta}. The parameter $q$ corresponds to the number of subspace iterations.}
\label{fig:ext}
\end{figure}

\subsection{Singular Values} We now consider the accuracy of the singular values. We use the same test matrices and the remaining parameters are kept fixed. The computed singular values are plotted against the upper and lower bounds. We make the following general observations:
{ 
\begin{itemize}
\item For the large singular values, both the upper and lower bounds are qualitatively good for all the examples that we tested.
\item As the number of iterations increase, the singular values are computed more accurately and are close to the upper bounds (the exact singular values). However, for indices close to the target rank, the lower bounds are are not tight. The bounds get tighter as the number of iterations $q$ increase. 
\item The bounds for the singular values quantitatively better than the bounds for the canonical angles. 
\end{itemize} 
}
\begin{figure}[!ht]%
\centering
\subfigure[GapSmall, $q=0$]{%
\label{fig:svsa1}%
\includegraphics[scale=0.24]{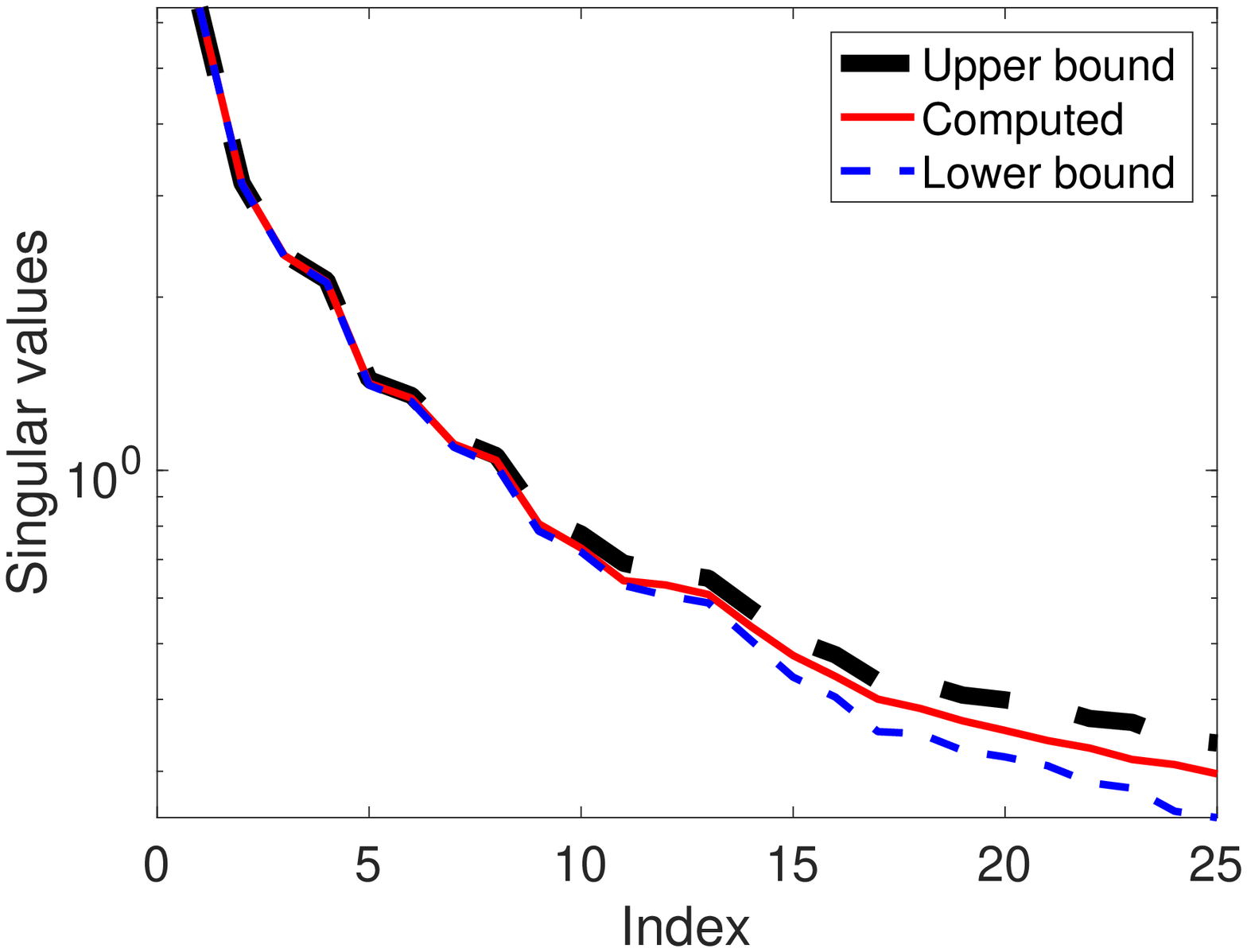}}%
\subfigure[GapSmall, $q=1$]{%
\label{fig:svsb1}%
\includegraphics[scale=0.24]{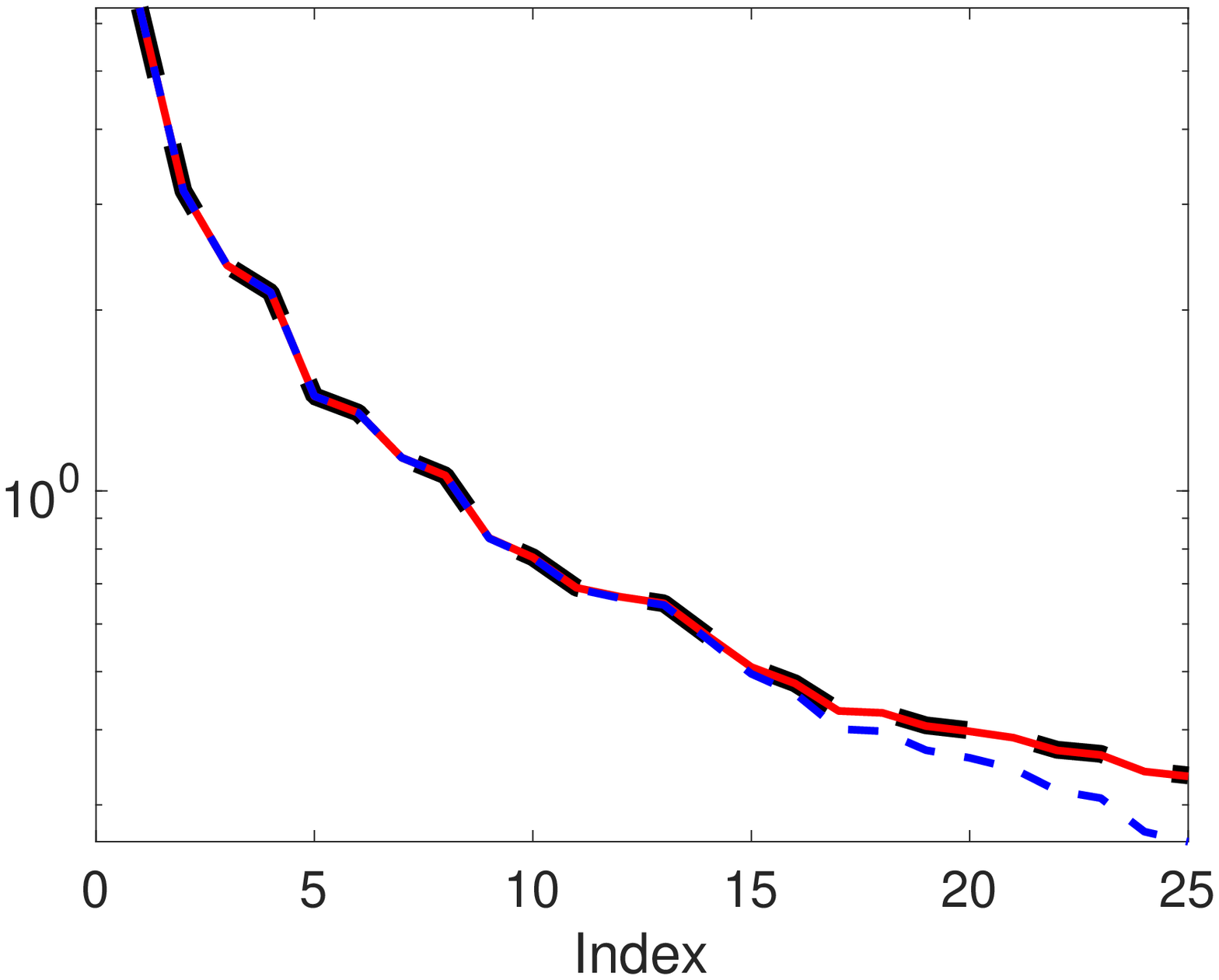}}%
\subfigure[GapSmall, $q=2$]{%
\label{fig:svsc1}%
\includegraphics[scale=0.24]{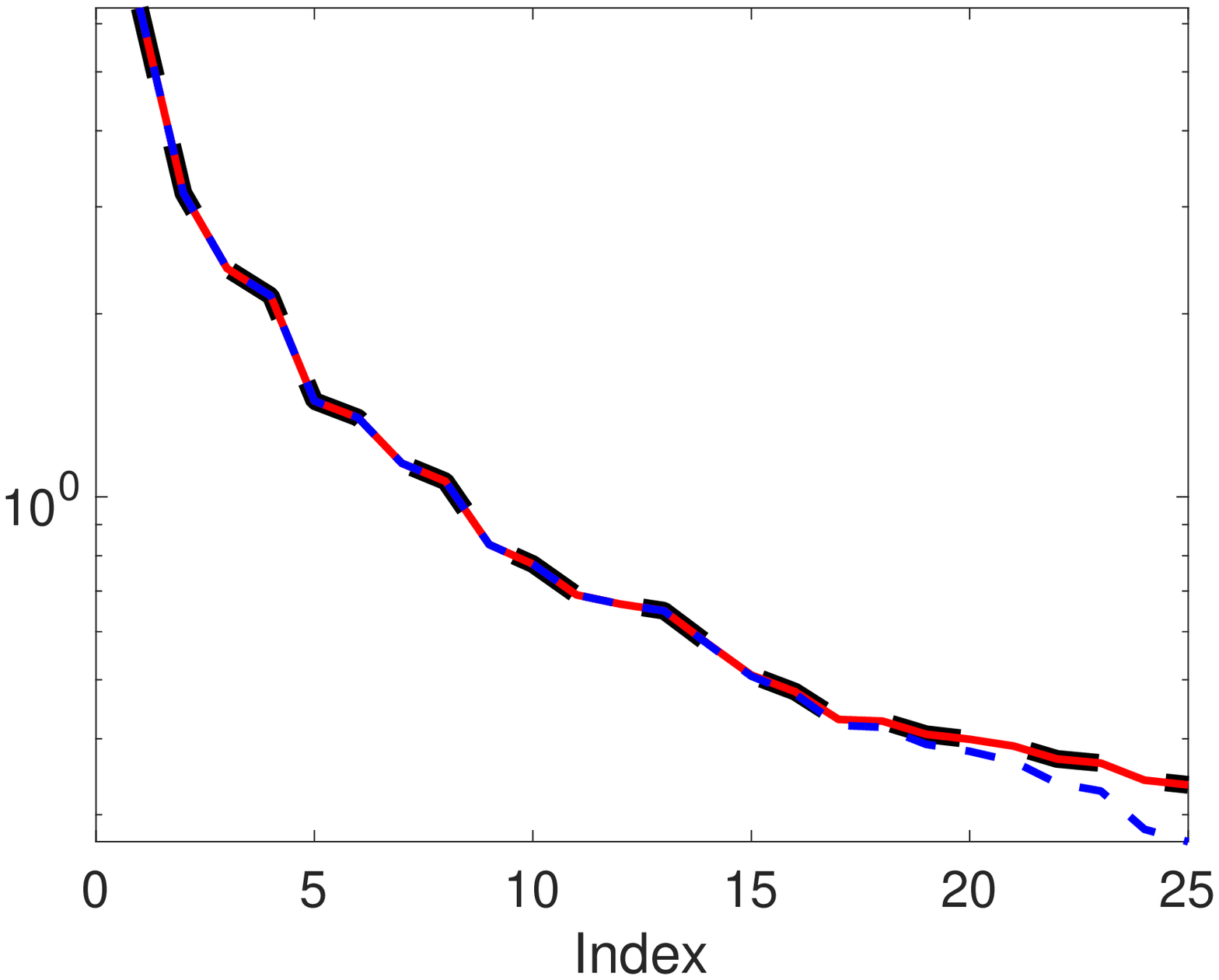}}%
\\
\subfigure[NoiseMedium, $q=0$]{%
\label{fig:svsa2}%
\includegraphics[scale=0.24]{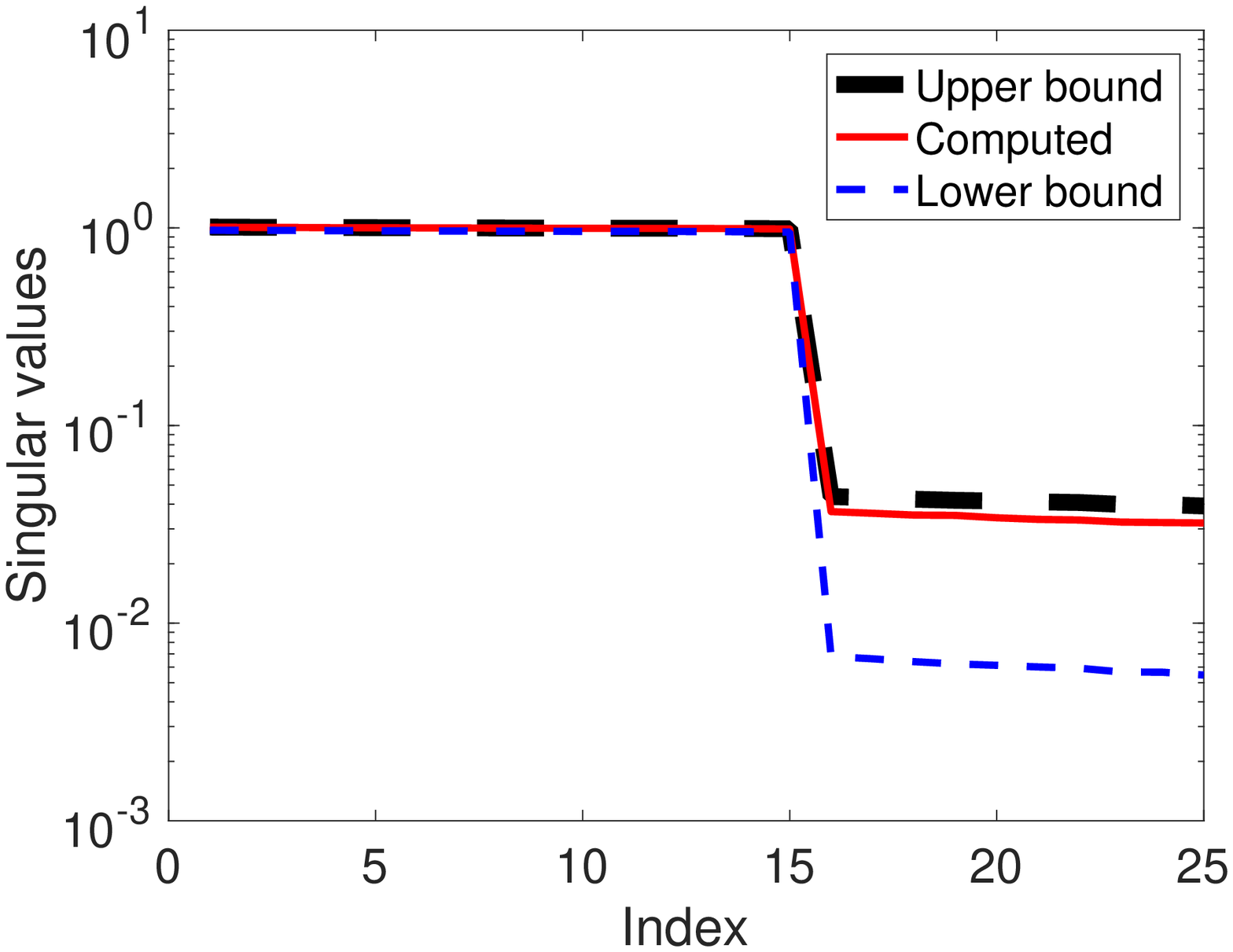}}%
\subfigure[NoiseMedium, $q=1$]{%
\label{fig:svsb2}%
\includegraphics[scale=0.24]{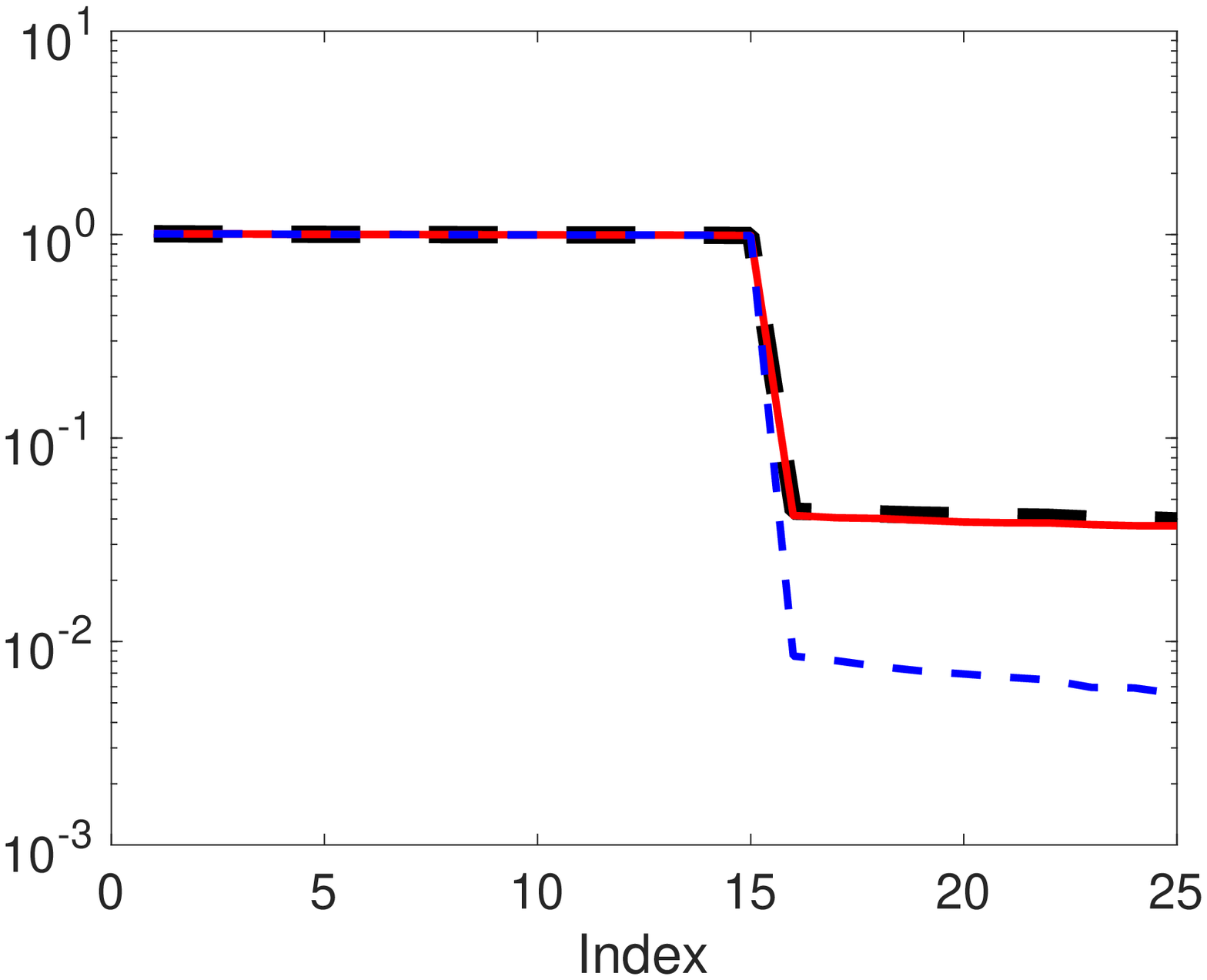}}%
\subfigure[NoiseMedium, $q=2$]{%
\label{fig:svsc2}%
\includegraphics[scale=0.24]{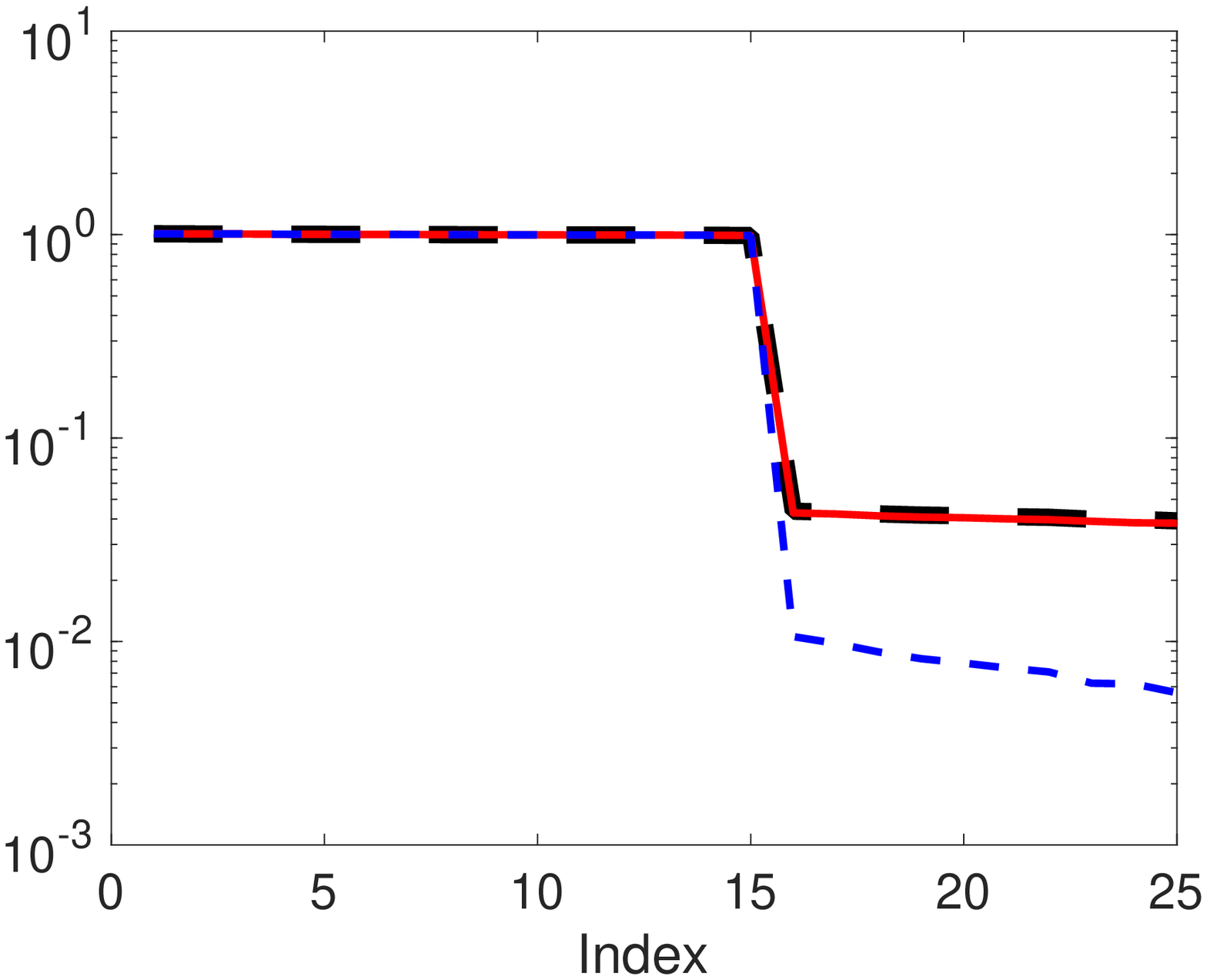}}%
\\
\subfigure[DecayMedium, $q=0$]{%
\label{fig:svsa3}%
\includegraphics[scale=0.24]{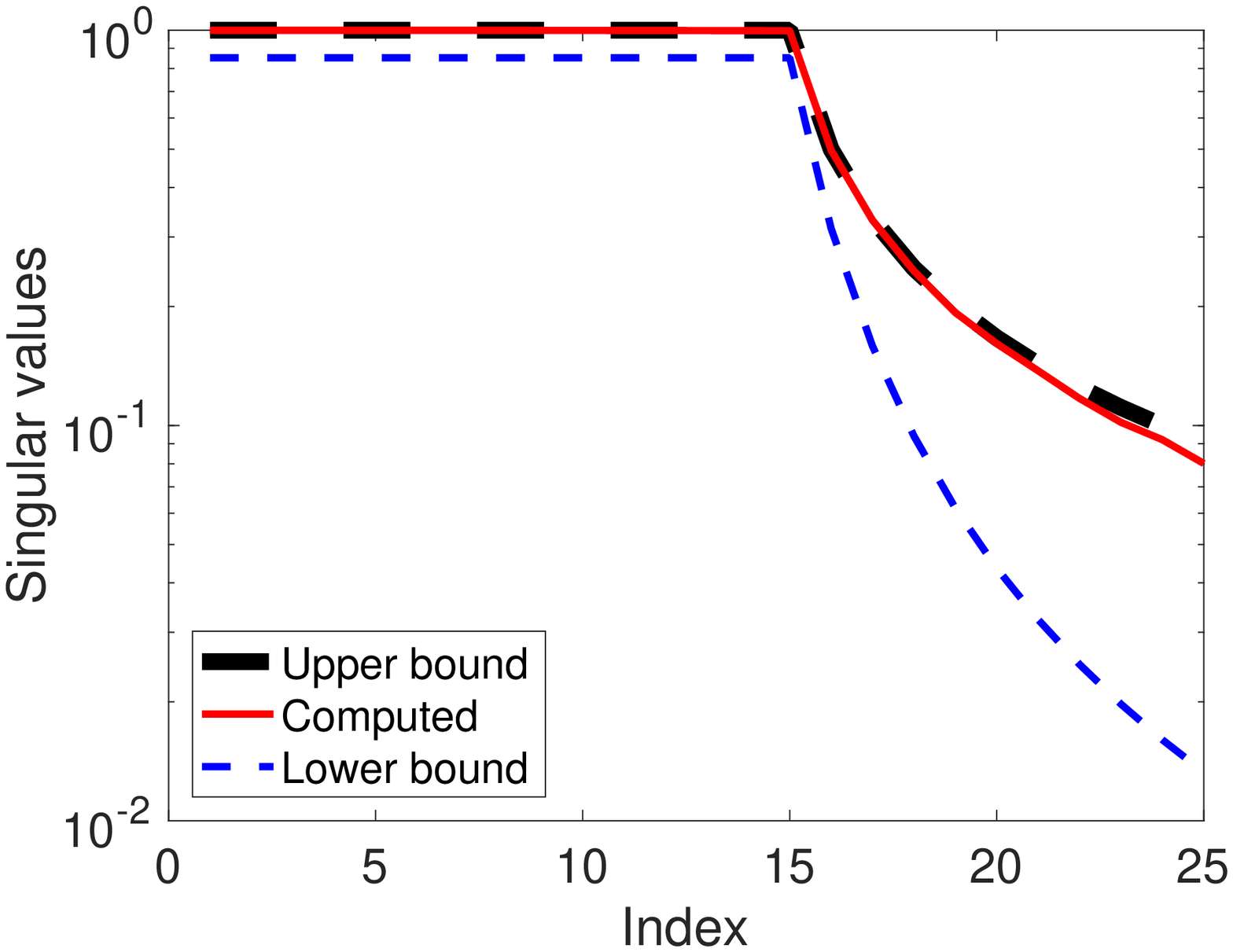}}%
\subfigure[DecayMedium, $q=1$]{%
\label{fig:svsb3}%
\includegraphics[scale=0.24]{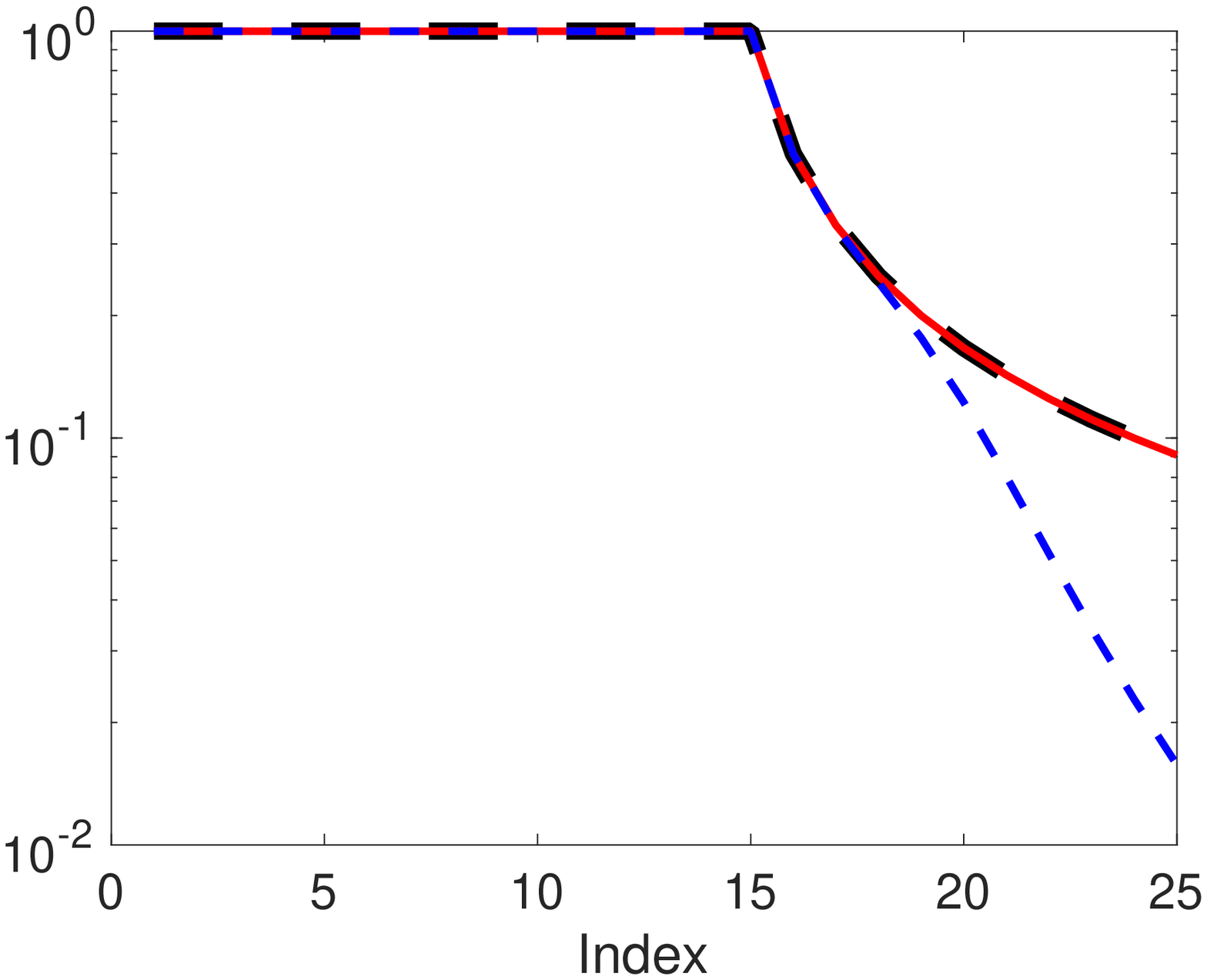}}%
\subfigure[DecayMedium, $q=2$]{%
\label{fig:svsc3}%
\includegraphics[scale=0.24]{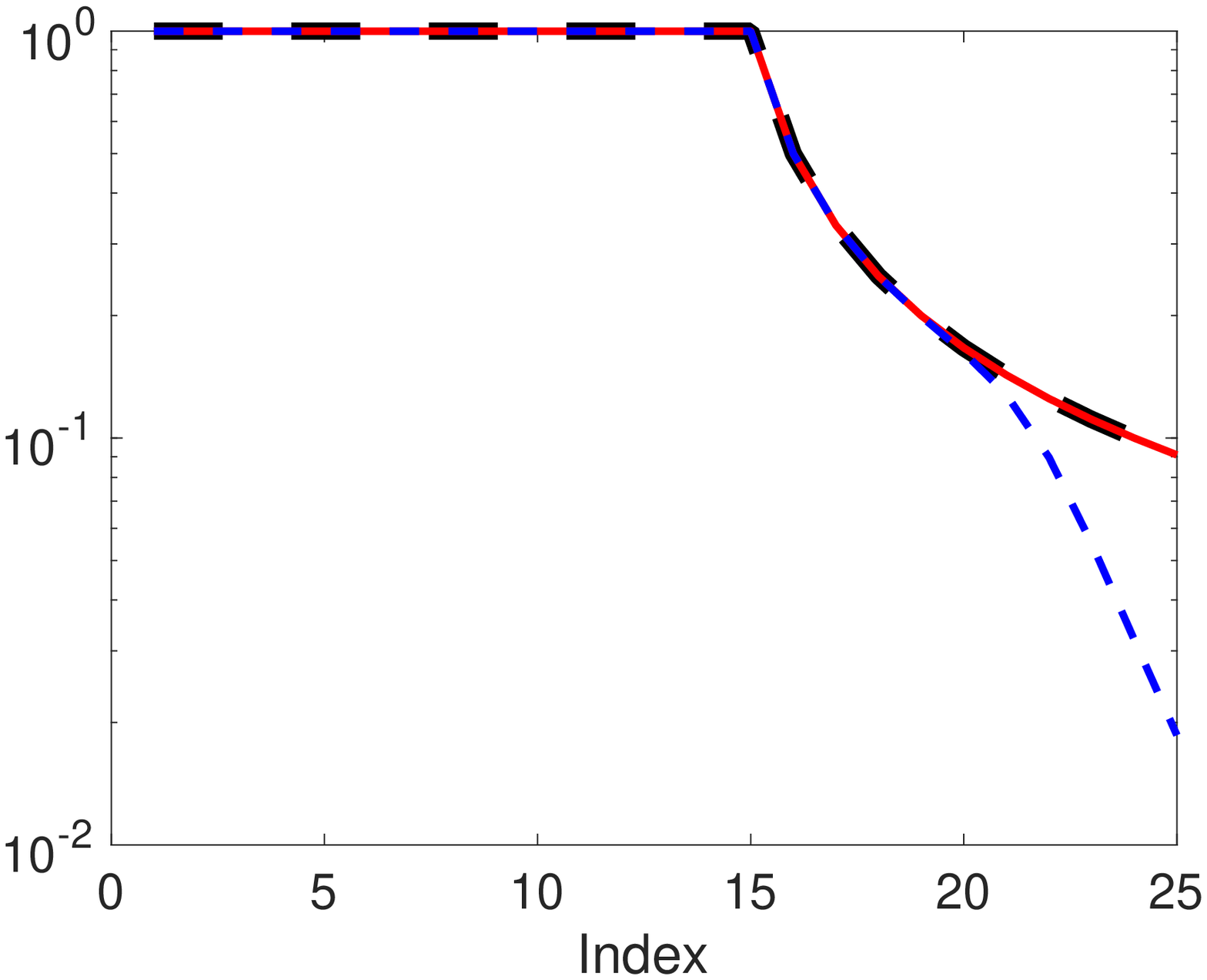}}%
\caption{Plots of the singular values. The test matrices were described in \cref{ssec:test}. The target rank $k=25$ and an oversampling parameter of $\rho=20$ was chosen for all the experiments. The solid lines black and blue lines correspond to the upper and lower bounds respectively, the dashed red lines correspond to bounds obtained using \cref{thm:sigma}. The parameter $q$ corresponds to the number of subspace iterations.}
\end{figure}

{

We now make observations specific to the test examples: 
\begin{description}
\item [1. GapSmall] In these examples, the large singular values are captured accurately. As the number of iterations increase, both the lower bound and the approximate singular values approach the true singular values (upper bound). For GapMedium and GapLarge, the bounds were much more accurate.
\item [2. NoiseMedium]  There is a qualitatively different behavior before and after indices $15-16$. The upper and lower bounds are tight before index $15$, but only the upper bound is tight after index $16$. The lower bound significantly under-predicts the singular values.
\item [3. DecayMedium] Similar to the previous example, the lower bounds are good before index $15$, and improve with number of iterations after index $15$.
\end{description}
 }

\section{Acknowledgments}
The author would like to thank Ilse C.F.\ Ipsen and Andreas Stathopoulos for helpful conversations. He would also like to acknowledge Ivy Huang for her help with the figures.

\bibliographystyle{abbrv}
\bibliography{references}
\end{document}